\documentclass[reqno,11pt]{amsart}

\usepackage{times,amsmath,cancel,stmaryrd,graphicx}
\usepackage{amsfonts,enumitem,amssymb,xcolor,mathrsfs}
\usepackage[colorlinks=true]{hyperref}
\usepackage{cleveref,enumerate}
 \usepackage[T1]{fontenc}
\usepackage{lmodern}
\usepackage{cite}
\usepackage{mathrsfs}

\numberwithin{equation}{section}

\newcommand{\R}{\mathbb{R}}
\newcommand{\C}{\mathbb{C}}

\newcommand{\N}{\mathbb{N}}
\newcommand{\G}{\Gamma}

\newcommand{\K}{\mathbb{K}}

\newcommand{\bT}{\mathbb{T}}

\newcommand{\kL}{\mathcal{L}}

\newcommand{\ve}{\varepsilon}
\newcommand{\rd}{\mathrm{d}}

\newcommand{\dhr}{\mathrel{\lhook\joinrel\relbar\kern-.8ex\joinrel\lhook\joinrel\rightarrow}}

\colorlet{halfblue}{blue!40!black!30!red}

\newtheorem{thm}{Theorem}[section]
\newtheorem{prop}[thm]{Proposition}
\newtheorem{lemma}[thm]{Lemma}

\newtheorem{rem}[thm]{Remark}

\theoremstyle{remark}

\begin{document}

\title[Linearized Stability of Non-Isolated Equilibria]{Linearized Stability of Non-Isolated Equilibria of Quasilinear Parabolic Problems in Interpolation Spaces}

\author{Bogdan--Vasile Matioc}
\address{Fakult\"at f\"ur Mathematik, Universit\"at Regensburg,   93053 Regensburg, Deutschland.}
\email{bogdan.matioc@ur.de}

\author{Christoph Walker}
\address{Leibniz Universit\"at Hannover\\
Institut f\"ur Angewandte Mathematik\\
Welfengarten 1\\
30167 Hannover\\
Germany}
\email{walker@ifam.uni-hannover.de}

\thanks{Partially supported by the RTG 2339
 ``Interfaces, Complex Structures, and Singular Limits'' of the German Science Foundation (DFG)}

\begin{abstract}
The stability of non-isolated equilibria to
 quasilinear parabolic problems of the form
$u' = A(u)u + f(u)$ is established in interpolation spaces (and thus extending previous results relying on maximal regularity). The approach allows full flexibility in choosing the interpolation methods  and requires only low regularity assumptions on the semilinear part $f$. Applications to concrete problems are presented, including the capillarity-driven Hele--Shaw problem and the fractional mean curvature flow.
\end{abstract}


\maketitle

\pagestyle{myheadings}

\markboth{\sc{B.-V.~Matioc \& Ch. Walker }}{Linearized Stability of Non-Isolated Equilibria}

\section{Introduction and Main Results}

In this paper we study stability properties of equilibrium solutions to quasilinear parabolic problems of the form
	\begin{equation}\label{CP}
		u'=A(u)u+f(u),\quad t>0,\qquad u(0)=u^0,
	\end{equation} 
assuming that $A(u)$ is an unbounded operator on a Banach space $E_0$ with domain $E_1$ which generates an analytic semigroup  on $E_0$.

 A well-established and powerful tool for proving 
exponential stability or instability of (isolated) stationary solutions is the principle of linearized stability, where
these properties are inferred from the sign of the real parts of the spectral values of the corresponding linearized operator  at the equilibrium.
Owing to its significance in applications, such abstract results have been developed in the parabolic setting within a variety of frameworks and by means of different techniques, see~\cite{DaPL88,D89,G88,Lu85,L95,PF81,P02,PSW18,MW20, MSW25} for an incomplete list of related works.

In certain applications, however, the set of equilibrium solutions is not discrete but instead forms a finite-dimensional smooth manifold. In some cases, conserved quantities (flow invariants) single out a unique equilibrium, which can then be identified as the omega-limit set of a given initial value by applying the principle of linearized stability to an appropriate restriction of the problem.

A standard—though technically involved—approach to analyzing the stability of stationary solutions in situations where flow invariants cannot be fully exploited is based on center manifold theory.
This theory is well established in the analysis of nonlinear parabolic evolution problems, see, e.g.,~\cite{Carr82,L95, HD81, DaPL88, BHL00, LPS08, LPS06, MA91, GS95, C09}.

More recently,  working within continuous, H\"older, or $L_p$-maximal regularity frameworks, the principle of linearized stability has been substantially  extended in~\cite{PS16, PSZ9a, PSZ09c}.
These works treat the situation in which the set of stationary solutions forms, in a neighborhood of an equilibrium~$u_*$, 
a finite-dimensional ${\rm C}^1$-manifold whose dimension coincides with that of the kernel of the linearized operator at $u_*$.
The resulting generalized principle of linearized stability yields  stability of $u_*$ and exponential convergence of solutions starting sufficiently close to $u_*$ 
toward a (possibly different) stationary solution near $u_*$, provided that $0$ is a semi-simple eigenvalue of the linearized operator and that the remainder of the spectrum lies in the open left half-plane.
This generalized principle is particularly valuable in applications, as it provides sharp convergence results under minimal regularity and spectral assumptions, 
and has been successfully applied to a wide range of nonlinear models, including quasilinear parabolic equations, free-boundary problems, fluid dynamics, phase transition models, and geometric evolution equations.

In this research we establish a similar result as in~~\cite{PS16, PSZ9a, PSZ09c}  without assuming a maximal property. 
Instead, we work in general interpolation spaces with full flexibility with respect to the choice of interpolation methods.
A key ingredient of our analysis is the concept of evolution operators, which is used to reformulate the evolution problem~\eqref{CP} as a fixed-point problem for~$u$ in suitable interpolation spaces between $E_0$ and $E_1$.
In particular, our first main result, see Theorem~\ref{MT1} below, provides linearized stability of equilibria for quasilinear problems within the well-established parabolic theory of~\cite{Am93}. We thus also extend the results of
~\cite{CDM25}, where the semilinear case was considered (see Remark~\ref{R:aMT1}).
In addition, we extend this theory  in Theorem~\ref{MT2} to the case of inclusions
$\mathrm{dom}(f)\subsetneq \mathrm{dom}(A)$ for the domains of the function~$u\mapsto f(u)$ and the quasilinear term~$u\mapsto A(u)$ for which the well-posedness and stability theory developed in~\cite{MRW25, MSW25} is formulated in phase spaces lying between $\mathrm{dom}(f)$ and $\mathrm{dom}(A)$. 
 Hence, this lowers the regularity assumption on $f$ and includes a limiting critical case which, in some applications, is a function space that is scaling invariant for the underlying evolution equation.  
 
The relevance of our abstract results for applications is illustrated in Example~\ref{Ex1}-Example~\ref{Ex3} in the context of the Hele-Shaw problem with surface tension, the fractional mean curvature flow (for periodic graphs), and a quasilinear evolution problem in critical scaling-invariant spaces.\medskip

Throughout this paper we assume that  $E_0$ and $E_1$ are Banach spaces over $\K\in\{\R,\C\}$ with continuous and dense embedding $E_1\hookrightarrow E_0$. For each~$\theta \in (0,1)$ we fix an arbitrary admissible interpolation functor~$(\cdot,\cdot)_\theta$  of
exponent $\theta$ (see \cite[I.~Section~2.11]{LQPP}) and define~${E_\theta := (E_0,E_1)_\theta}$ as the corresponding interpolation space with norm $\|\cdot\|_\theta$.

\subsection*{Linearized Stability of Non-Isolated Equilibria}

We begin with the case  of equal domains $\mathrm{dom}(f)=\mathrm{dom}(A)$. To this end, we fix  exponents
\begin{subequations}\label{assA}
\begin{equation}\label{as1}
 0<\gamma\leq \beta<\alpha<1
\end{equation}
 and assume, for some open subset $O_\beta$ of $E_\beta$, that
 \begin{equation}\label{as2A}
(A,f)\in {\rm C}^{1}\big(O_\beta,\mathcal{H}(E_1,E_0) \times E_\gamma\big),
\end{equation}
\end{subequations}
where $\mathcal{H}(E_1,E_0)$  denotes the open subset of $\mathcal{L}(E_1,E_0)$ consisting of generators of strongly continuous analytic semigroups on $E_0$.

We first recall from \cite{Am93} (see also \cite[Theorem~1.1]{MW20}) that the quasilinear Cauchy problem~\eqref{CP} is locally well-posed in~$O_\alpha$, where 
$O_\theta := O_\beta \cap E_\theta$ for $\theta \in (\beta,1]$:

\begin{thm}\label{T1}
Assume~\eqref{assA}. Then, for each $u^0\in O_\alpha$, the Cauchy problem \eqref{CP} possesses a unique maximal classical solution
\begin{equation} \label{regu}
u=u(\cdot;u^0)\in {\rm C}^1\big((0,t^+),E_0\big)\cap {\rm C}\big((0,t^+),E_1\big)\cap  {\rm C}^{\alpha-\theta} \big([0,t^+),E_\theta\big)
\end{equation}
for every $0\le \theta\le \alpha$,   where  $t^+=t^+(u^0)\in(0,\infty]$ is the maximal existence time.
Moreover, if the orbit $u([0,t^+(u^0));u^0)$ is relatively compact in $O_\alpha$, then $t^+(u^0)=\infty$. 
\end{thm}

We build upon this well-posedness result to investigate stability properties of equilibria that form a finite-dimensional manifold. Our approach is inspired by \cite{PSZ9a} and follows the strategy developed therein.
To state the precise result, we need some preparation and notation.\medskip

Denote  the set of equilibria of~\eqref{CP} by $\mathcal{E}$, that is,
\[
\mathcal{E} := \{ u \in O_1 \,:\, A(u)u + f(u) = 0 \}.
\]
We study the asymptotic behavior of solutions $u(\cdot;u^0)$ with initial data $u^0 \in O_\alpha$
 in a neighborhood of a fixed equilibrium $u_* \in \mathcal{E}$ on which we impose suitable assumptions specified below.
We assume that there exist a neighborhood~$U$ of $0$ in $\R^m$, a function~${\Psi \in \mathrm{C}^2(U,E_1)}$, and a constant $r>0$ such that
\begin{equation}\label{E}
\Psi(U)= \mathcal{E}\cap \mathbb{B}_{E_1}(u_*,r),\qquad \Psi(0)=u_*,\qquad \mathrm{rank}(\partial \Psi(0))=m.
\end{equation}
Assumption~\eqref{E} guarantees that, in a neighborhood of $u_*$, the set of equilibria $\mathcal{E}$ forms an $m$-dimensional $\mathrm{C}^2$-manifold.

Introducing $v:=u-u_*$, we linearize problem \eqref{CP} at $u_*$ to obtain the form
\begin{equation}\label{CPv}
v'=A_*(v)v+g(v),\quad t>0,\qquad v(0)=u^0-u_*,
\end{equation}
where,  given $v\in O_\beta^*:=O_\beta-u_*$, we defined
$$
A_*(v)w:=A(u_*+v)w+(\partial A(u_*)[w]) u_* +\partial f(u_*)w, \qquad w\in E_1,
$$
and
$$
g(v):=f(u_*+v)-f(u_*)-\partial f(u_*)v+\big(A(u_*+v)-A(u_*)-(\partial A(u_*)v)\big) u_*.
$$
It follows from  \eqref{as2A}  that
\begin{equation*}
(A_*,g)\in {\rm C}^{1}\big(O_\beta^*,\mathcal{L}(E_1,E_0) \times E_0\big)
\end{equation*}
with
\begin{equation}\label{g}
g(0)=0 ,\qquad \partial g(0)=0.
\end{equation}  
Moreover, setting $\psi(\xi):=\Psi(\xi)-u_*$ for $\xi\in U$, the first condition of~\eqref{E} entails that
\begin{equation}\label{equil}
A_*(\psi(\xi))\psi(\xi)+g(\psi(\xi))=0, \qquad \xi\in U.
\end{equation}
Differentiating this identity at $\xi=0$ and using that $\psi(0)=0$ and $\partial g(0)=0$ readily implies that $A_*(0)\partial\psi(0)=0$. 
That is, $T_{u_*}\mathcal{E}\subset \mathrm{ker}(A_*(0))$, with $T_{u_*}\mathcal{E}= \mathrm{rg}(\partial \Psi(0))$ denoting the tangent space of $\mathcal{E}$ at $u_*$, and
$$
A_*(0)=A(u_*)+(\partial A(u_*)[\cdot ]) u_* +\partial f(u_*).
$$
In fact, we shall further assume that
\begin{subequations}\label{assB}
\begin{equation}\label{tangent}
T_{u_*}\mathcal{E}= \mathrm{ker}(A_*(0)),
\end{equation}
that $0$ is a semi-simple eigenvalue of $A_*(0)$, i.e.,
\begin{equation}\label{semisimple}
\text{$\mathrm{rg}(A_*(0))$ is closed and}\ E_0=\mathrm{ker}(A_*(0)) \oplus \mathrm{rg}(A_*(0)),
\end{equation}
and that
\begin{equation}\label{spectrum}
\sigma(A_*(0))\setminus\{0\}\subset [\mathrm{Re}\, z<0].
\end{equation}
\end{subequations}
In particular, $0$ is an isolated spectral point of $\sigma(A_*(0))$ which yields the disjoint union 
$$\sigma(A_*(0))=\{0\}\cup \sigma_s(A_*(0)),$$  
$$
-\omega_0:=\sup\{\mathrm{Re}\, \lambda\,:\, \lambda\in \sigma_s(A_*(0))\}<0.
$$
We point out that assumptions~\eqref{E} and \eqref{assB} express the fact that the equilibrium solution~$u_*$ is {\it normally stable}~\cite{PS16,PSZ9a,PSZ09c}.

Let $P\in \mathcal{L}(E_0)$ be the spectral projection  onto $\mathrm{ker}(A_*(0))$ corresponding to the spectral set~$\{0\}$ of~$A_*(0)$,  see e.g. \cite[Appendix A]{L95}, and denote by 
 $Q:=1-P$ the associated projection of~$E_0$ onto $\mathrm{rg}(A_*(0))$. 
 Then $\mathrm{rg}(P)=\mathrm{ker}(A_*(0))\subset E_1$  and~${\mathrm{rg}(Q)=\mathrm{rg}(A_*(0))}$.
Moreover, setting $E_\theta^0:=P(E_\theta)$ and $E_\theta^s:=Q(E_\theta)$, we obtain the topological sum decompositions
\begin{subequations}\label{Astable}
\begin{equation}\label{Astable1}
E_\theta =E_\theta^0\oplus E_\theta^s, \qquad \theta\in [0,1],
\end{equation}
and, since $\mathrm{rg}(P)=\mathrm{ker}(A_*(0))= E_1^0$, we have in fact that 
\begin{equation}\label{Astable1'}
E_\theta^0=\mathrm{ker}(A_*(0))= E_1^0, \qquad \theta\in [0,1].
\end{equation}
The decomposition \eqref{Astable1} reduces~$A_*(0)$  into
\begin{equation}\label{Astable2}
A_*(0)=0\oplus A_*^s(0)
\end{equation}
 and, in view of~\cite[Appendix A]{L95}, the stable part $A_*^s(0)$ satisfies
\begin{equation}\label{Astable3}
A_*^s(0)\in \mathcal{H}(E_1^s,E_0^s)\qquad\text{and}\qquad \sigma(A_*^s(0))= \sigma_s(A_*(0))\subset [\mathrm{Re}\, z\le -\omega_0].
\end{equation}
\end{subequations}

Finally, in order to state our stability result we require a further, technical assumption: There are $\zeta\in [0,1)$ and 
Banach spaces $F$ and $G$ such that
\begin{subequations}\label{assC}
\begin{equation}\label{assC1}
\qquad E_{\zeta}\hookrightarrow G, \qquad    P \in  \mathcal{L}(F,P(F)) , \qquad P(F)\hookrightarrow E_1^0,
\end{equation}
and 
\begin{equation}\label{110b}
 \|A(w)-A(u_*)\|_{\mathcal{L}(G,F)}\le L\|w-u_*\|_{E_\beta}, \qquad w\in O_\beta.
\end{equation}
\end{subequations}
 Clearly, we may actually  assume without loss of generality that $\alpha<\zeta<1$. 

\begin{rem}
Assumption~\eqref{assC} constitutes an essential ingredient of our analysis as it compensates for the lack of maximal regularity properties of the quasilinear part. In fact, it allows us to handle the semilinear term~$g$ in the reformulation~\eqref{CPv} of~\eqref{CP}, which is only $E_0$-valued.
 We emphasize, however, that this assumption imposes no significant restriction in many applications and is easily checked,  see Example~\ref{Ex1}-Example~\ref{Ex3}, and in particular~Proposition~\ref{PPP}.
In fact, in applications the operator $P\in \mathcal{L}(F,P(F))$ is usually an extension of the projection~$P\in \mathcal{L}(E_0)$ and $P(F)=E_1^0$.
\end{rem}

 We also note that,  since $E_1^0$ is  finite-dimensional, all norms on~$E_1^0$ are equivalent, 
and we equip $E_1^0$ with $\|\cdot\|_1$. \medskip
 
We shall prove that the normally stable equilibrium $u_*$  for~\eqref{CP} is stable.
  Moreover, we shall show that for any initial value~$u^0$ sufficiently close to $u_*$ in $E_\alpha$
   the corresponding maximal solution $u(\cdot;u^0)$ is globally defined and converges exponentially fast to an  equilibrium~$\hat{u}_* \in \mathcal{E}$. 
The asymptotic limit~$\hat{u}_* \in \mathcal{E}$ is uniquely determined by the initial value~$u^0$, although it may vary with $u^0$.
 
\begin{thm}\label{MT1}
Assume~\eqref{assA} and, for the equilibrium solution $u_*\in\mathcal{E}$, assume~\eqref{E},~\eqref{assB}, and~\eqref{assC}. Then $u_*$ is stable in $E_\alpha$. Moreover, for any given $\omega\in(0,\omega_0)$ there exist constants~$\delta>0$ and~${K\geq 1}$ such that for each  $u^0\in \mathbb{B}_{E_\alpha}(u_*,\delta)$, the maximal solution~$u(\cdot;u^0)$ to~\eqref{CP} is globally defined, and there exists a unique $\hat u_*\in \mathcal{E}$ such that
\begin{equation}\label{des:expstab}
\|u(t;u^0)-\hat u_*\|_{E_\alpha}\leq K e^{-\omega t}\|u^0-u_*\|_{E_\alpha},\qquad t\geq 0.
\end{equation}
\end{thm}

The proof of Theorem~\ref{MT1} is given in Section~\ref{Sec:2}, and we refer to Example~\ref{Ex1} and Example~\ref{Ex2} for two applications of this result.\medskip

 Of course, Theorem~\ref{MT1} also applies to the semilinear case where $A(u)=A$ is independent of $u$. In this case, however, less restrictive assumptions are needed as observed in~\cite{CDM25}:

\begin{rem}\label{R:aMT1}\,
{\bf (i)} In \eqref{E} we assume, for simplicity, that the set of equilibria forms an $m$-dimensional ${\rm C}^{2}$-manifold in a neighborhood of $u_*$. 
However, it suffices  that this set is a ${\rm C}^{2-}$-manifold, or even a ${\rm C}^{1+\delta}$-manifold for some $\delta \in (0,1)$. 
 In the semilinear case, a~${\rm C}^{1}$-manifold is actually sufficient (see Footnote~\ref{foot}).\smallskip

\noindent{\bf (ii)} In the semilinear case, condition~\eqref{as1} may be replaced by the slightly more general requirement 
\[
0 \le \gamma < \alpha < 1,
\]
i.e. the choice $\gamma = 0$ is admissible, see~\cite[Theorem~1.2]{MW_PRSE} and the proof of Theorem~\ref{MT1}. 
\smallskip

\noindent{\bf (iii)} Assumption~\eqref{assC} is automatically satisfied in the semilinear case. 

\end{rem}

\subsection*{Linearized Stability of Non-Isolated Equilibria in Time-Weighed Spaces}
In some applications, the semilinear term $u\mapsto f(u)$ in \eqref{CP} requires higher regularity of $u$ than the quasilinear part~$u\mapsto A(u)$. 
In such situations, Theorem~\ref{MT1} no longer provides optimal results. 
Using the well-posedness theory developed in \cite{MW_PRSE, MRW25}, we now extend Theorem~\ref{MT1} to this setting, following an approach analogous to that in \cite[Theorem~1.3]{MSW25}.\medskip

With interpolation spaces $E_\theta$, $\theta\in[0,1]$, as defined above, we fix 
\begin{subequations}\label{Hypo}
\begin{equation}\label{Hypo1}
0<\gamma< \beta< \xi<1, \qquad  q\ge 1,
\end{equation}
and impose on the quasilinear part in \eqref{CP} the same assumptions as before, namely 
\begin{equation}\label{HypoA1}
A\in {\rm C}^1\big(O_\beta ,\mathcal{H}(E_1 ,E_0) \big).
\end{equation}
The main difference to the previous setting concerns the semilinear part satisfying
\begin{equation}\label{Hypof1}
f\in {\rm C}^1(O_\xi ,E_\gamma)
\end{equation}
and thus  possibly requiring more regularity as $\xi>\beta$.
Moreover, we additionally assume  for each~$R>0$ that there exists a constant~$C(R)>0$  such that
\begin{equation}\label{Hypof2}
\begin{aligned}
&\|f(u)-f(v)\|_{E_\gamma}\\
&\le  C(R)\big[1+\|u\|_{\xi}^{q-1}+\|v\|_{\xi}^{q-1}\big]\big[\big(1+\|u\|_{\xi}+\|v\|_{\xi}\big) \|u-v\|_{\beta}+\|u-v\|_{\xi}\big]
\end{aligned}
\end{equation}
 for  $u,v\in O_\xi\cap \bar{\mathbb{B}}_{E_\beta}(0,R)$.
The quantitative estimate~\eqref{Hypof2}, involving the parameter~$q\ge 1$ that measures  
the growth of~$f$ with respect to $E_\xi$-terms,  
leads naturally to the use of time-weighted spaces, see~\eqref{tws} below, when studying the well-posedness  
of~\eqref{CP}. In \cite{MW_PRSE, MRW25}  this approach is developed under slightly milder   assumptions than stated in~\eqref{Hypo}, 
whereas \cite{PW17,PSW18,LeCroneSimonett20,HvNWV_III} address the problem in suitable maximal regularity frameworks.
The advantage of employing time weights for  the well-posedness of \eqref{CP}  lies in the use of phase spaces $E_\alpha$ with parameters  
$\alpha\in(\beta,\xi)$ for which the semilinearity $f$ need not be defined.  
To be more precise, we define the critical value $\alpha_{\rm crit}<\xi$ by
\begin{equation}\label{HypoX}
\alpha_{\rm crit}:=\frac{q\xi-1-\gamma}{q-1} \quad \text{if } q>1,
\qquad 
\alpha_{\rm crit}:=-\infty \quad \text{if } q=1,
\end{equation}
and consider initial data $u^0\in O_\alpha$ with
\begin{equation}\label{HypoXX}
\alpha_{\rm crit}\le \alpha\in(\beta,\xi).
\end{equation}
\end{subequations}

In the critical case $\alpha=\alpha_{\rm crit}\in(\beta,\xi)$,  
the well-posedness theory developed in~\cite{MRW25} requires the following additional assumption:  
there exists an interpolation functor $\{\cdot,\cdot\}_{\alpha/\xi}$  
of exponent $\alpha/\xi$, and for each $\eta\in\{\alpha,\beta,\xi\}$  
there  exists an interpolation functor $\{\cdot,\cdot\}_{\gamma/\eta}$  
of exponent $\gamma/\eta$, such that
\begin{equation}\label{Y}
E_\alpha \doteq \{E_0,E_\xi\}_{\alpha/\xi},
\qquad 
E_\gamma \doteq \{E_0,E_\eta\}_{\gamma/\eta},
\quad 
\eta\in\{\alpha,\beta,\xi\}.
\end{equation}
Property ~\eqref{Y} is not really restrictive in applications as it is automatically satisfied if, for each~$\theta\in\{\gamma,\beta,\alpha,\xi\}$, the interpolation functor~${(\cdot,\cdot)_\theta}$ is either the complex interpolation functor~$[\cdot,\cdot]_\theta$,  or the continuous interpolation functor~${(\cdot,\cdot)_{\theta,\infty}^0}$,
  or the real interpolation functor~$(\cdot,\cdot)_\theta=(\cdot,\cdot)_{\theta,p}$  with parameter $p\in [1,\infty]$ (see e.g.~\cite[I.Remarks~2.11.2~(b)]{LQPP}).\medskip
  
The local well-posedness of \eqref{CP} in $O_\alpha$ under hypotheses~\eqref{Hypo} and~\eqref{Y} 
was established in \cite{MRW25,MW_PRSE} and is summarized in \cite[Theorem~1.2]{MSW25}. 
We recall this result:

\begin{thm}\label{T2}
Assume~\eqref{Hypo}-\eqref{Y}. Then, for each $u^0\in O_\alpha$, the Cauchy problem~\eqref{CP} possesses a unique maximal classical solution $u=u(\cdot;u^0)$ such that
\begin{subequations}\label{solab}
\begin{equation}\label{sola}
\begin{aligned}
u &\in {\rm C}^1\big((0,t^+),E_0\big) \cap {\rm C}\big((0,t^+),E_1\big)\cap {\rm C}\big([0,t^+),O_\alpha\big) \cap {\rm C}^{\alpha-\beta}\big([0,t^+),E_\beta\big)
\end{aligned}
\end{equation}
and
\begin{equation}\label{solb}
\lim_{t\to 0} t^{\xi-\alpha}\|u(t)\|_{E_\xi} = 0,
\end{equation}  
\end{subequations}
where $t^+=t^+(u^0)\in(0,\infty]$ denotes the maximal existence time.
\end{thm}

We can now state the stability result  for normally stable equilibria within this framework:
\begin{thm}\label{MT2}
Assume~\eqref{Hypo}-\eqref{Y} and, for the equilibrium~${u_*\in\mathcal{E}}$, assume  \eqref{E},~\eqref{assB}, and~\eqref{assC}. 
Then $u_*$ is  stable in $E_\alpha$. Moreover, for any given $\omega\in(0,\omega_0)$ there exist constants~$\delta>0$ and~${K\geq 1}$ such that for each  $u^0\in \mathbb{B}_{E_\alpha}(u_*,\delta)$, the maximal solution~$u(\cdot;u^0)$ to~\eqref{CP} is globally defined, and there exists a unique $\hat u_*\in \mathcal{E}$ such that
\begin{equation}\label{des:expstab2}
\|u(t;u^0)-\hat u_*\|_{E_\alpha}+t^{\xi-\alpha}\|u(t;u^0)-\hat u_*\|_{E_\xi}\leq K e^{-\omega t}\|u^0-u_*\|_{E_\alpha},\qquad t\geq 0.
\end{equation}
\end{thm}

The proof of this result is given in Section~\ref{Sec:3} and an application  in Example~\ref{Ex3}.  
We note that assumption~\eqref{Y} is not needed in the non-critical case $ \alpha_{\rm crit}<\alpha\in (\beta,\xi)$ :

\begin{rem}\label{RRa1}
In the noncritical case $\alpha_{\rm crit}<\alpha\in (\beta,\xi)$ one may in fact omit assumption~\eqref{Y} in Theorem~\ref{MT2}.  
The claim still holds (with a similar proof) if estimate~\eqref{des:expstab2} is replaced by
\begin{equation*}
\|u(t;u^0)-\hat u_*\|_{E_\alpha}+t^{\mu}\|u(t;u^0)-\hat u_*\|_{E_\xi}\le M e^{-\omega t} \|u^0\|_{E_\alpha}, \qquad t> 0,
\end{equation*}
for some fixed  (but arbitrary) $\mu>\xi-\alpha$.  
\end{rem}

\section{Proof of Theorem~\ref{MT1}}\label{Sec:2}
 
In this section we provide the proof of  Theorem~\ref{MT1}.
We use the notation introduced earlier and adapt the proof of \cite[Theorem~2.1]{PSZ9a} to our setting.
Moreover, we use the fact that
 $A\in\mathcal{H}(F_1,F_0)$ for some densely embedded Banach spaces $F_1\hookrightarrow F_0$ if and only if there exist constants $\kappa\geq 1$ and $\omega>0$ such  that $A\in\mathcal{H}(F_1,F_0;\kappa,\omega)$, meaning that  
 \[
 \omega-A\in\mathcal{L}is(F_1,F_0):=\{T\in\mathcal{L}(F_1,F_0)\,:\, \text{$T$ is bijective}\}
 \]
 and
 \[
\|(\lambda-A)x\|_{F_0}\geq |\lambda|\, \|x\|_{F_0}+\|x\|_{F_1},\qquad x\in F_1,\ {\rm Re}\ \lambda\geq \omega.
 \]
We refer to \cite{LQPP} for details.

\medskip

\begin{proof}[Proof of Theorem~\ref{MT1}]  We divide the lengthy proof into four steps.\medskip

\noindent{\bf (i)}  We first represent the manifold $\mathcal{E}$  in a neighborhood of $u_*$ as a translated graph. 
To this end, define $h:U\to E_1^0$ by 
$h(\xi):=P\psi(\xi)$ for $\xi\in U$ and note that
$$\partial h(0)=P\partial\psi(0)\in \mathcal{L}is(\R^m,E_1^0)$$ 
due to~\eqref{E},  \eqref{assB}, and~\eqref{Astable1'}. The inverse function theorem then ensures that the mapping~$h$ is a ${\rm C}^{2}$-diffeomorphism in a neighborhood of $0$ and hence invertible. 
Thus, there exists $r_0>0$ such that
$$
h^{-1}\in {\rm C}^2\big(\mathbb{B}_{E_1^0}(0,r_0),U\big), \qquad h^{-1}(0)=0.
$$ 
Set $\Phi(x):=\psi\big(h^{-1}(x)\big)$ for $x\in \mathbb{B}_{E_1^0}(0,r_0)$. Then $\Phi\in {\rm C}^{2}\big(\mathbb{B}_{E_1^0}(0,r_0),E_1\big)$ satisfies $\Phi(0)=0$ and, recalling~\eqref{assA},
$$
\big\{\Phi(x)+u_*\,:\, x\in \mathbb{B}_{E_1^0}(0,r_0)\}=\mathcal{E}\cap W
$$
for some neighborhood $W$ of $u_*$ in $O_1$. Since 
$$
P\Phi(x)= P\psi\big(h^{-1}(x)\big)=x, \qquad x\in \mathbb{B}_{E_1^0}(0,r_0),
$$
by definition of $h$, we have
$$
\Phi(x)= P\Phi(x)+Q\Phi(x)=x+Q\Phi(x), \qquad x\in \mathbb{B}_{E_1^0}(0,r_0).
$$
Consequently, defining  $\phi(x):=Q\Phi(x)$ for $x\in \mathbb{B}_{E_1^0}(0,r_0)$, we obtain $\phi\in {\rm C}^{2}\big(\mathbb{B}_{E_1^0}(0,r_0),E_1^s\big)$ and deduce from   $\mathrm{rg}(P)=\mathrm{rg}(\partial\psi(0))$
that
\begin{subequations}\label{phi}
\begin{equation}\label{phi1}
 \phi(0)=0, \qquad \partial\phi(0)=Q\partial\psi(0)\partial h^{-1}(0)=0,
\end{equation}
and
\begin{equation}\label{phi2}
 \left\{x+\phi(x)+u_*\,:\, x\in \mathbb{B}_{E_1^0}(0,r_0)\right\}=\mathcal{E}\cap W
\end{equation}
\end{subequations}
for some neighborhood $W$ of $u_*$ in $O_1$. 
Therefore, the manifold $\mathcal{E}$ can indeed be represented in a neighborhood of $u_*$ as the graph of the function $\phi$ translated by $u_*$. 

Making $r_0>0$ smaller, if necessary, we may assume due to~\eqref{phi1} that
\begin{equation}\label{n9}
\|\partial \phi(x)\|_{\mathcal{L}(E_1^0,E_1^s)}\le 1   \quad\text{and}\quad \|\phi(x)\|_{E_1^s}\le \|x\|_{E_1^0}, \qquad  x\in \mathbb{B}_{E_1^0}(0,r_0),
\end{equation}
in particular, $\phi$ is uniformly Lipschitz continuous, and
\begin{equation}\label{n9b}
x+ \phi(x) +y\in  O_\beta^*, \qquad x\in \mathbb{B}_{E_1^0}(0,r_0),\ y\in \mathbb{B}_{E_\beta^s}(0,r_0).
\end{equation}
We also note from \eqref{equil} and the identities $\psi(h^{-1}(x)) =\Phi(x)=x+\phi(x)$  that
\begin{align}\label{n8}
A_*\big(x+\phi(x)\big)\big(x+\phi(x)\big)+g\big(x+\phi(x)\big)=0, \qquad x\in \mathbb{B}_{E_1^0}(0,r_0).
\end{align}

\noindent{\bf (ii)}
Next, we derive the corresponding Cauchy problem for the projections of the translated solution $u - u_*$ onto $E_1^0$ and  $E_1^s$, respectively. 
This Cauchy problem (see~\eqref{xeqs}-\eqref{yeqs} below) will play a crucial role in the subsequent analysis.

More precisely, let $u^0 \in O_\alpha$ be such that $P(u^0 - u_*) \in \mathbb{B}_{E_1^0}(0,r_0)$.
Let $u = u(\cdot;u^0)$ be the maximal solution to \eqref{CP} provided by Theorem~\ref{T1} and set $v := u - u_*$. 
Then there exists $T \in (0,t^+(u^0))$ sufficiently small, such that
$Pv(t) \in \mathbb{B}_{E_1^0}(0,r_0)$ for all  $t \in [0,T).$
 Hence, 
\begin{equation}\label{defxy}
\begin{aligned}
x(t)&:=Pv(t)=P(u(t)-u_*) \in \mathbb{B}_{E_1^0}(0,r_0),\\
y(t)&:=Q(v(t))-\phi(x(t))=Q(u(t)-u_*)-\phi(x(t))
\end{aligned}
\end{equation}
are well-defined for $ t\in[0,T)$ and satisfy
\[
x(t)+\phi(x(t))+y(t)=u(t)-u_*=v(t)\in O_\beta^*,\qquad t\in[0,T).
\]
 Moreover, since $\phi\in {\rm C}^{2}\big(\mathbb{B}_{E_1^0}(0,r_0),E_1^s\big)$ is uniformly Lipschitz continuous, we  infer from~\eqref{regu} that 
\begin{equation}\label{regxy}
\begin{aligned} 
&x\in {\rm C}^1\big((0,T),E_1^0\big)\cap  {\rm C}^{\alpha-\theta} \big([0,T),E_1^0\big),\\
&y\in {\rm C}^1\big((0,T),E_0^s\big)\cap {\rm C}\big((0,T),E_1^s\big)\cap  {\rm C}^{\alpha-\theta}\big([0,T),E_\theta^s\big)
\end{aligned}
\end{equation}
for $\theta\in[0,\alpha]$, and, by~\eqref{CPv}, the pair $(x,y)$ solves the evolution problem
\begin{align*}
x'&=S(x,y), \quad t>0,   && x(0)=x^0:= P(u^0-u_*),\\
y'&=R(x,y), \quad   t>0, && y(0)=y^0:=Q(u^0-u_*)-\phi(P(u^0-u_*)),
\end{align*}
in $E_1^0\times E_0^s$, where
\begin{align*}
S(x,y)&:=PA_*\big(x+\phi(x)+y\big)\big(x+\phi(x)+y\big)+Pg\big(x+\phi(x)+y\big), \\\
R(x,y)&:=QA_*\big(x+\phi(x)+y\big)\big(x+\phi(x)+y\big)+Qg\big(x+\phi(x)+y\big)-\partial\phi(x) S(x,y).
\end{align*}
To unfold the structure of this system, we  slightly abuse notation and simply write $(x,y)$ neglecting the time variable when appropriate. 
We first use identity~\eqref{n8} to write
\begin{align*}
S(x,y)&=P\big[A_*\big(x+\phi(x)+y\big)-A_*\big(x+\phi(x)\big)\big]\big(x+\phi(x)\big)\\
&\quad+P\big[g\big(x+\phi(x)+y\big)-g\big(x+\phi(x)\big)\big]+P A_*\big(x+\phi(x)+y\big)y.
\end{align*}
Noticing then that $A_*(v)=A(u_*+v)-A(u_*)+A_*(0)$ for $v\in O_\beta^*$ together with \eqref{Astable2}-\eqref{Astable3} implies that
$$
PA_*(v)(x+z)=P\big[A(u_*+v)-A(u_*)\big](x+z), \qquad v\in O_\beta^*,\ x\in E_1^0, \ z\in E_1^s,
$$
and recalling that $\phi(x), y\in E_1^s$, we derive that 
\begin{subequations}\label{S}
\begin{align}\label{decs}
S(x,y)=S_1(x,y)+S_2(x,y),
\end{align}
where
\begin{equation}
\begin{split}
S_1(x,y)&:=P\big[A(u_*+x+\phi(x)+y)-A(u_*+x+\phi(x))\big](x+\phi(x))\\
&\qquad+P\big[g(x+\phi(x)+y)-g(x+\phi(x))\big]
\end{split}
\end{equation}
for $x\in \mathbb{B}_{E_1^0}(0,r_0)$ and  $y\in \mathbb{B}_{E_\beta^s}(0,r_0)$, and
\begin{align}
S_2(x,y)&:=P\big[A(u_*+x+\phi(x)+y)-A(u_*)\big]y
\end{align}
for $x\in \mathbb{B}_{E_1^0}(0,r_0)$ and  $y\in E_1^s\cap\mathbb{B}_{E_\beta^s}(0,r_0)$.
\end{subequations}

Similarly, we deduce from~\eqref{n8} first that
\begin{align*}
R(x,y)&=Q A_*(x+\phi(x)+y)y+Q\big[A_*(x+\phi(x)+y)-A_*(x+\phi(x)\big](x+\phi(x))\\
&\quad+Q\big[g(x+\phi(x)+y)-g(x+\phi(x))\big]-\partial\phi(x) S(x,y)
\end{align*}
and then, using 
$$
QA_*(v)(x+z)=Q\big[A(u_*+v)-A(u_*)\big](x+z)+A_*^s(0)z, \qquad v\in O_\beta^*,\ x\in E_1^0, \ z\in E_1^s,
$$
 that
\begin{align*}
R(x,y)&=A_*^s(0)y+Q\big[A(u_*+x+\phi(x)+y)-A(u_*)\big]y\\
&\quad+Q\big[A(u_*+x+\phi(x)+y)-A(u_*+x+\phi(x))\big](x+\phi(x))\\
&\quad+Q\big[g(x+\phi(x)+y)-g(x+\phi(x))\big]-\partial\phi(x) S(x,y).
\end{align*}
Consequently, we obtain from the representation of $S(x,y)$ in~\eqref{S} that
\begin{align*}
R(x,y)&=\mathcal{A}(x,y)y+q(x,y)
\end{align*}
for $x\in \mathbb{B}_{E_1^0}(0,r_0)$ and  $y\in E_1^s\cap\mathbb{B}_{E_\beta^s}(0,r_0)$,
where we set
\begin{equation}\label{mcA}
\begin{split}
\mathcal{A}(x,y)&:=A_*^s(0)+Q\big[A(u_*+x+\phi(x)+y)-A(u_*)\big]\\
&\qquad -\partial\phi(x)P\big[A(u_*+x+\phi(x)+y)-A(u_*)\big]
\end{split}
\end{equation}
and
\begin{equation}\label{q}
\begin{split}
q(x,y)&:=Q\big[A(u_*+x+\phi(x)+y)-A(u_*+x+\phi(x))\big](x+\phi(x))\\
&\qquad+Q\big[g(x+\phi(x)+y)-g(x+\phi(x))\big]\\
&\qquad -\partial\phi(x) P\big[A(u_*+x+\phi(x)+y)-A(u_*+x+\phi(x))\big](x+\phi(x))\\
&\qquad -\partial\phi(x)P\big[g(x+\phi(x)+y)-g(x+\phi(x))\big].
\end{split}
\end{equation}
Summarizing,  the pair $(x,y)$ solves the Cauchy problem 
\begin{align}
x'&=S(x,y), \quad t>0,   && x(0)=x^0:= P(u^0-u_*),\label{xeqs}\\
y'&=\mathcal{A}(x,y)y+q(x,y), \quad   t>0, && y(0)=y^0:=Q(u^0-u_*)-\phi(P(u^0-u_*)),\label{yeqs}
\end{align}
in $E_1^0\times E_0^s$, where
\begin{align*}
&S: \mathbb{B}_{E_1^0}(0,r_0)\times\big(E_1^s\cap\mathbb{B}_{E_\beta^s}(0,r_0)\big)  \to  E_1^0,\\
&\mathcal{A}: \mathbb{B}_{E_1^0}(0,r_0)\times\mathbb{B}_{E_\beta^s}(0,r_0)  \to  \mathcal{L}(E_1^s,E_0^s)
\end{align*} 
 are defined in~\eqref{S} and~\eqref{mcA}, respectively, while 
\begin{align*}
&q: \mathbb{B}_{E_1^0}(0,r_0)\times\mathbb{B}_{E_\beta^s}(0,r_0)  \to  E_0^s
\end{align*}
 is defined in~\eqref{q}. Obviously, the functions $S$, $\mathcal{A}$, and $q$ are continuously differentiable.
 
 In fact, it readily follows from~\eqref{as2A}, \eqref{g}, and~\eqref{n9}
that, for any $\eta\in (0,r_0)$ (to be chosen  below, see step {\bf (iv)}), there exists $r(\eta)\in (0,\eta)$ such that
\begin{equation}\label{e2}
\|q(x,y)\|_{E_0^s}\le \eta \|y\|_{E_\beta^s}, \qquad x\in \bar{\mathbb{B}}_{E_1^0}(0,r(\eta)), \ y\in \bar{\mathbb{B}}_{E_\beta^s}(0,r(\eta)),
\end{equation}
and
\begin{equation} \label{e1x}
\|S_1(x,y)\|_{E_1^0}\le \eta \|y\|_{E_\beta^s}, \qquad x\in \bar{\mathbb{B}}_{E_1^0}(0,r(\eta)), \ y\in \bar{\mathbb{B}}_{E_\beta^s}(0,r(\eta)),
\end{equation}
where we recall that all norms on $E_1^0$ are equivalent. 
Concerning $S_2$, we invoke assumption~\eqref{assC} to deduce that there exists a constant $C>0$ such that
\begin{align*}
\|S_2(x,y)\|_{E_1^0}&\le C\|S_2(x,y)\|_{P(F)} \leq C \|A(u_*+x+\phi(x)+y)-A(u_*)\|_{\mathcal{L}(G,F)} \|y\|_{G}\\
&\le \eta \|y\|_{E_\zeta^s}
\end{align*}
for  $x\in \bar{\mathbb{B}}_{E_1^0}(0,r(\eta))$ and $y\in E_\zeta^s\cap\bar{\mathbb{B}}_{E_\beta^s}(0,r(\eta))$.
In fact, since $E_\zeta^s\hookrightarrow E_\beta^s$, we may thus assume in view of \eqref{decs} and~\eqref{e1x}  that
\begin{align}\label{e1}
\|S(x,y)\|_{E_1^0}\le \eta \|y\|_{E_\zeta^s}, \qquad x\in \bar{\mathbb{B}}_{E_1^0}(0,r(\eta)), \ y\in E_\zeta^s\cap\bar{\mathbb{B}}_{E_\beta^s}(0,r(\eta)).
\end{align}
 Finally, recalling~\eqref{mcA} and making $r(\eta)\in(0,\eta)$ smaller, if necessary, we infer from~\eqref{as2A} and~\eqref{n9} that
\begin{equation}\label{AA}
\|\mathcal{A}(x,y)-A_*^s(0)\|_{\mathcal{L}(E_1^s,E_0^s)}\le \eta , \qquad x\in \bar{\mathbb{B}}_{E_1^0}(0,r(\eta)), \ y\in \bar{\mathbb{B}}_{E_\beta^s}(0,r(\eta)).
\end{equation}

To simplify the notation, we set
$$
\bar{\mathbb{B}}_\beta(r(\eta)):=\bar{\mathbb{B}}_{E_1^0}(0,r(\eta))\times \bar{\mathbb{B}}_{E_\beta^s}(0,r(\eta))
$$
for $\eta \in (0,r_0)$.
\medskip

\noindent{\bf (iii)} 
Fix $\bar\alpha \in ( \max\{\beta,\alpha+\zeta-1\},\alpha)$  and define $\rho := \alpha - \bar\alpha \in (0, \min\{\alpha-\beta,1-\zeta\})$.  
For any~$T \in (0,\infty]$, we introduce the set
\begin{align*}
\mathcal{M}_\eta(T)
:= \Big\{ (x,y) & \in  \mathrm{C}\big([0,T), \bar{\mathbb{B}}_\beta(r(\eta))\big) \,:\, \\
& \|x(t)-x(s)\|_{E_1^0} + \|y(t)-y(s)\|_{E_\beta^s}
\le N |t-s|^\rho, \quad t,s \in  [0,T)\Big\},
\end{align*}
where $N>0$ is to be chosen suitably (see \eqref{defN} below). Given $(x,y)\in \mathcal{M}_\eta(T)$, we associate with  $\mathcal{A}(x,y)$ the unique parabolic evolution operator
$U_{\mathcal{A}(x,y)}$ on $E_0^s$ with regularity subspace $E_1^s$  and provide some estimates that will be important for our purpose.

To begin with,  fix  $\omega\in (0,\omega_0)$  and set $4\ve:=\omega_0-\omega>0$. Invoking \cite[II.Proposition~1.4.2]{LQPP}, we infer from~\eqref{Astable3} and~\eqref{AA} (making $r(\eta)>0$ smaller, if necessary) that there exist constants~$\kappa\ge 1$ and $L\ge 1$ such that 
$$
\omega_0-\ve +\mathcal{A}(x,y) \in \mathcal{H}(E_1^s,E_0^s;\kappa,\ve), \qquad (x,y)\in \bar{\mathbb{B}}_\beta(r(\eta)),
$$
and, recalling~\eqref{n9} and the fact that $\phi\in {\rm C}^{2}\big(\mathbb{B}_{E_1^0}(0,r_0),E_1^s\big)$\footnote{\label{foot}At this point we use the Lipschitz continuity of~$\partial\phi$ and, implicitly,  the assumption that the equilibria form a ${\rm C}^2$-manifold. In the semilinear case, $\mathcal{A}$ is independent of $(x,y)$ and we obviously do not need the Lipschitz continuity of $\partial\phi$ so that a ${\rm C}^1$-manifold is sufficient.},
$$
\|\mathcal{A}(x,y)-\mathcal{A}(\bar x,\bar y)\|_{\mathcal{L}(E_1^s,E_0^s)}\le L \big(\|x-\bar x||_{E_1^0}+\|y-\bar y||_{E_\beta^s}\big),  \qquad (x,y), (\bar x,\bar y)\in \bar{\mathbb{B}}_\beta(r(\eta)).
$$
These properties  ensure, for each $(x,y)\in \mathcal{M}_\eta(T)$, that 
\begin{subequations}\label{xx}
\begin{equation}
\mathcal{A}(x,y)\in {\rm C}^\rho\big([0,T),\mathcal{L}(E_1^s,E_0^s)\big)
\end{equation}
with
\begin{equation}
\| \mathcal{A}(x(t),y(t))-\mathcal{A}(x(s),y(s))\|_{\mathcal{L}(E_1^s,E_0^s)}\le LN\vert t-s\vert^\rho, \qquad t,s\in  [0,T),
\end{equation}
and
\begin{equation} 
\omega_0-\ve +\mathcal{A}(x(t),y(t)) \in \mathcal{H}(E_1^s,E_0^s;\kappa,\ve), \qquad t\in  [0,T).
\end{equation}
\end{subequations}

Properties~\eqref{xx} allow us to invoke the stability estimates of \cite[Section~II.5]{LQPP} for the family
$\{\mathcal{A}(x,y)\,:\, (x,y)\in \mathcal{M}_\eta(T)\}$.
Let $c_0(\rho)>0$ be the constant from \cite[II.Theorem~5.1.1]{LQPP} and choose $N>0$ such that 
\begin{equation}\label{defN}
c_0(\rho)(LN)^{1/\rho}= \ve.
\end{equation} 
Note that  
\begin{equation}\label{nu}
-\nu:= c_0(\rho)(LN)^{1/\rho}-\omega_0+2\ve= -\omega_0+3\ve=-\omega-\ve <0.
\end{equation}
In  then follows from \cite[II.Theorem~5.1.1, II.Lemma~5.1.3]{LQPP},~\eqref{xx}, and~\eqref{nu} that, for each $(x,y)\in \mathcal{M}_\eta(T)$, there exists a unique parabolic evolution operator $U_{\mathcal{A}(x,y)}$ on $E_0^s$ with regularity subspace $E_1^s$. Moreover, there exists a constant $M\ge 1$ (depending on the constants $LN$, $\rho$, $\kappa$, $\omega_0$, and~$\ve$, but independent of $T\in(0,\infty]$) such that, 
for $ 0\le s\le t< T$,
\begin{equation}\label{ev0} 
\|U_{\mathcal{A}(x,y)}(t,s)\|_{\mathcal{L}(E_\theta^s)}+ (t-s)^{\theta-\vartheta_0}\|U_{\mathcal{A}(x,y)}(t,s)\|_{\mathcal{L}(E_\vartheta^s,E_\theta^s)}  \leq  Me^{-(\omega+\ve) (t-s)},
\end{equation}
provided that $0\le \vartheta_0\le \vartheta\le  \theta\le 1$ with $\vartheta_0<\vartheta$ if $0<\vartheta<  \theta< 1$. Moreover, it follows from
\cite[II.Theorem~5.3.1]{LQPP}  that
$$
w:=\left[t\mapsto U_{\mathcal{A}(x,y)}(t,0)w^0+\int_0^t U_{\mathcal{A}(x,y)}(t,\tau) p(\tau)\, \rd\tau\right]\in {\rm C}^{\rho}( [0,T),E_{\bar\alpha}^s)
$$
and (making $M\ge 1$ larger, if necessary)
\begin{equation}\label{ww}
\|w(t)-w(\tau)\|_{E_{\bar\alpha}^s} \le M(t-\tau)^{\rho}\big(\|w^0\|_{E_{\alpha}^s}+\|p\|_{L_\infty((0,t),E_0^s)}\big), \qquad 0\le \tau\le  t< T,
\end{equation}
whenever $w^0\in E_{\alpha}^s$ and $p\in L_{\infty,{\rm loc}}([0,T),E_0^s)$, where we recall that $\rho=\alpha-\bar\alpha$.
\medskip

\noindent{\bf (iv)}  Finally, we combine the reformulation \eqref{xeqs}–\eqref{yeqs} of \eqref{CP} with the stability estimates \eqref{ev0}–\eqref{ww} to finish the proof of Theorem~\ref{MT1}.

Let therefore $R>0$ be such that $\bar{\mathbb{B}}_{E_\alpha}(u_*,R)\subset O_\alpha$ and choose  $\eta\in (0,\min\{r_0,R/ (2 c_{1,\alpha})\})$ and $\delta_0\in (0,1)$ small enough such that 
\begin{subequations}
\label{choice}
\begin{align}
& c_{\alpha,\beta}M\eta \int_0^\infty \big(1+\sigma^{-\alpha}+\sigma^{-\zeta} \big)e^{-\ve \sigma/2} \,\rd \sigma\le \frac{1}{2},\qquad \delta_0\le \frac{r(\eta)}{4M} ,\label{choice1}\\
& c_{\bar \alpha,\beta}M\big(1+2M\eta\big)\delta_0 \le\frac{N}{4}, \qquad \frac{M \eta \delta_0}{\rho} \left(\sup_{\sigma>0} \big(\sigma^{1-\rho}+\sigma^{1-\rho-\zeta}\big)e^{-(\omega  +\ve/2) \sigma}\right)\le\frac{N}{4},\label{choice2}
\end{align}
\end{subequations}
where $r(\eta)\in (0,\eta)$ is as in {\bf (ii)}, and where $c_{\vartheta,\theta}>0$ denotes the norm of the continuous embedding $E_\vartheta^s \hookrightarrow E_\theta^s$ for $0\le \theta\le\vartheta\le 1$ (without loss of generality we assume that~${c_{\vartheta,\theta}\ge 1}$).
Consider then $u^0\in \bar{\mathbb{B}}_{E_\alpha}(u_*,\delta)\subset O_\alpha$ with $\delta\in (0,R)$ small enough such that (recall that all norms on $E_1^0$ are equivalent)
\begin{align}\label{x}
\|x^0\|_{E_1^0}=\|P(u^0-u_*)\|_{E_1^0}\le \|P\|_{\mathcal{L}(E_\alpha,E_1^0)}\|u^0-u_*\|_{E_\alpha}\le \delta_0   <r_0,
\end{align}
where the last inequality  follows from $0<r(\eta)<\eta<r_0$  and \eqref{choice1}, and such that (using~\eqref{n9})
\begin{equation}\label{y} 
\begin{split}
\|y^0\|_{E_\beta^s}&\le c_{\alpha,\beta}\|y^0\|_{E_\alpha^s}=c_{\alpha,\beta}\|Q(u^0-u_*)-\phi(P(u^0-u_*))\|_{E_\alpha^s}\\
&\le c_{\alpha,\beta}\|Q\|_{\mathcal{L}(E_\alpha,E_\alpha^s)}\|u^0-u_*\|_{E_\alpha}+c_{\alpha,\beta}c_{1,\alpha}\|P(u^0-u_*)\|_{E_1^0}
\le \delta_0.
\end{split}
\end{equation}
 Recall that $u = u(\cdot;u^0)$ denotes the maximal solution to problem~\eqref{CP}  with maximal existence time~$t^+=t^+(u^0)$ provided by Theorem~\ref{T1} and that $(x,y)$ is introduced in~\eqref{defxy} satisfying~\eqref{regxy}. Due to \eqref{choice1}, \eqref{x}, and~\eqref{y} we can choose $T\in (0,t^+)$ sufficiently small such that
$$
\|x(t)\|_{E_1^0}\le 2\delta_0\le r(\eta) <r_0 ,\qquad \|y(t)\|_{E_\beta^s}\le c_{\alpha,\beta}\|y(t)\|_{E_\alpha^s}\le 2\delta_0\le r(\eta)
$$  
for $t\in  [0,T)$ and, in view of~\eqref{regxy},  we may further assume that there is a constant~${C>0}$ such that 
\begin{align*}
\|x(t)-x(\tau)\|_{E_1^0}+\|y(t)-y(\tau)\|_{E_\beta^s}\le  C T^{ \bar \alpha -\beta}\vert t-\tau\vert^\rho\le N \vert t-\tau\vert^\rho
\end{align*}
for $t,\tau\in  [0,T)$. Define
\begin{align*}
T_0 := \sup \left\{ T \in (0,t^+) \,:\,
\begin{minipage}{0.65\linewidth}
$\|x(t)\|_{E_1^0} \le r(\eta) <r_0, \
\|y(t)\|_{ E_\beta^s} \le c_{\alpha,\beta}\|y(t)\|_{E_\alpha^s} \le r(\eta)$,\\[1ex]
$\|x(t)-x(\tau)\|_{E_1^0}
 + \|y(t)-y(\tau)\|_{E_\beta^s}
 \le N |t-\tau|^\rho,$ \ $ t,\tau \in [0,T)$
\end{minipage}
\right\}.
\end{align*}
Then,  $T_0\in (0,t^+]$ and it follows from~\cite[II.Remarks~2.1.2]{LQPP}, \eqref{regxy}, \eqref{yeqs},   and~{\bf (iii)} that $y$ is given by the variation-of-constants formula
$$
y(t)=U_{\mathcal{A}(x,y)}(t,0) y^0+\int_0^t U_{\mathcal{A}(x,y)}(t,\tau) q\big(x(\tau),y(\tau)\big)\,\rd \tau, \qquad 0\le t<T_0.
$$
 For $\theta\in \{\alpha,\zeta\}$, let $\alpha_0=\alpha_0(\theta)$ be defined by $\alpha_0:=\alpha$ if $\theta=\alpha$ and  $\alpha_0:=0$ if $\theta=\zeta$ (recall that $\alpha<\zeta$). 
 Then, in virtue of~\eqref{e2} and~\eqref{ev0}, we have
\begin{align*}
e^{\omega t}\|y(t)\|_{E_\theta^s}&\le e^{\omega t}\|U_{\mathcal{A}(x,y)}(t,0)\|_{\mathcal{L}(E_\alpha^s,E_\theta^s)}\| y^0\|_{E_\alpha^s}\\
&\quad +e^{\omega t}\int_0^t \|U_{\mathcal{A}(x,y)}(t,\tau)\|_{\mathcal{L}(E_0^s,E_\theta^s)} \left\|q\big(x(\tau),y(\tau)\big)\right\|_{E_0^s}\,\rd \tau \\
&\le M t^{\alpha_0-\theta}e^{-\ve t} \| y^0\|_{E_\alpha^s} + M \eta  \int_0^t e^{-\ve(t-\tau)}(t-\tau)^{-\theta} e^{\omega\tau}\|y(\tau)\|_{E_\beta^s}\,\rd \tau
\end{align*}
for $t\in (0,T_0)$. On the one hand, for $\theta=\alpha$ we deduce  that
\begin{align*}
e^{(\omega +\ve/2)t}\|y(t)\|_{E_\alpha^s}\le M  \| y^0\|_{E_\alpha^s} + c_{\alpha,\beta}M \eta  \Big(\int_0^\infty \sigma^{-\alpha} e^{-\ve\sigma/2}\,\rd \sigma \Big)\sup_{\tau\in[0,t]}e^{(\omega+\ve/2)\tau}\|y(\tau)\|_{E_\alpha^s}
\end{align*}
for $t\in [0,T_0)$,  so that \eqref{choice1} then ensures
\begin{align}\label{yy}
 \|y(t)\|_{E_\alpha^s}\le  2 M e^{-(\omega+ \ve/2) t} \| y^0\|_{E_\alpha^s}, \qquad t\in [0,T_0).
\end{align}
Consequently, using again \eqref{choice1}  and~\eqref{y}, we get
\begin{align}\label{yyy}
\|y(t)\|_{E_\beta^s} &\le c_{\alpha,\beta} \|y(t)\|_{E_\alpha^s}\le  2 M c_{\alpha,\beta} \| y^0\|_{E_\alpha^s}\le 2M\delta_0\le \frac{r(\eta)}{2} ,\qquad t\in [0,T_0).
\end{align}
On the other hand, taking $\theta=\zeta$ and using \eqref{yy}, we  obtain
\begin{align*}
e^{(\omega+ \ve/2) t}\|y(t)\|_{E_{\zeta}^s}&\le \Big[M t^{-\zeta}  + 2 c_{\alpha,\beta} M^2 \eta    \Big(\int_0^\infty \sigma^{-\zeta} e^{-\ve \sigma/2}\,\rd \sigma\Big)\Big]\| y^0\|_{E_\alpha^s}
\end{align*}
and hence
\begin{align}\label{yyyy}
\|y(t)\|_{E_{\zeta}^s}&\le M\big(1+t^{-\zeta}\big)e^{-(\omega+ \ve/2) t} \| y^0\|_{E_\alpha^s} ,\qquad t\in (0,T_0),
\end{align}
due to \eqref{choice1}. Next, note from~\eqref{ww},  exploiting also~\eqref{e2},  \eqref{y}, and~\eqref{yyy}, that
\begin{align*}
\|y(t)-y(\tau)\|_{E_{\bar\alpha}^s} & \le M\big(\|y^0\|_{E_{\alpha}^s}+\|q(x,y)\|_{L_\infty((0,t),E_0^s)}\big)(t-\tau)^{\rho}\\
&\le  M\big(\|y^0\|_{E_{\alpha}^s}+\eta\|y\|_{L_\infty((0,t),E_\beta^s)}\big)(t-\tau)^{\rho}\\
&\le  M\big(1+2M\eta\big)\delta_0 (t-\tau)^{\rho}
\end{align*}
for $0\le \tau< t< T_0$, so that \eqref{choice2} implies that
\begin{align}\label{y5}
\|y(t)-y(\tau)\|_{E_{\beta}^s}\le c_{\bar\alpha,\beta}\|y(t)-y(\tau)\|_{E_{\bar\alpha}^s} \le \frac{N}{4} (t-\tau)^{\rho}, \qquad 0\le \tau< t< T_0.
\end{align}
We  now turn to equation~\eqref{xeqs} satisfied by $x$  and infer from \eqref{regxy},~\eqref{e1}, and~\eqref{yyyy}  that
\begin{equation}\label{xxx}
\begin{aligned}
\|x(t)-x(\tau)\|_{E_1^0}&\le \int_\tau^t \|S(x(\sigma),y(\sigma))\|_{E_1^0}\,\rd \sigma \le \eta\int_\tau^t \|y(\sigma)\|_{E_\zeta^s}\,\rd \sigma \\
&\le   M\eta \| y^0\|_{E_\alpha^s} \int_\tau^t \big(1+\sigma^{-\zeta}\big)e^{-(\omega +\ve/2)\sigma} \,\rd \sigma
\end{aligned}
\end{equation}
for $0\le \tau< t< T_0$. On the one hand,  we obtain from~\eqref{y} and~\eqref{xxx} that
\begin{align*}
\|x(t)-x(\tau)\|_{E_1^0}
&\le  M\eta  \| y^0\|_{E_\alpha^s}\left(\sup_{\sigma>0}\big(\sigma^{1-\rho}+\sigma^{1-\rho-\zeta}\big)e^{-(\omega +\ve/2) \sigma}\right)\int_\tau^t \sigma^{\rho-1} \,\rd \sigma\\
&\le \frac{ M\eta \delta_0}{\rho} \left(\sup_{\sigma>0}\big(\sigma^{1-\rho}+\sigma^{1-\rho-\zeta}\big)e^{-(\omega +\ve/2) \sigma}\right) \big(t^\rho-\tau^\rho\big)
\end{align*}
and hence, using~\eqref{choice2},  
\begin{align}\label{y6}
\|x(t)-x(\tau)\|_{E_1^0}&\le \frac{N}{4} (t-\tau)^\rho, \qquad 0\le \tau< t< T_0.
\end{align}
On the other hand, we deduce from~\eqref{xxx} combined with~\eqref{choice1}, \eqref{x}, and~ \eqref{y},  that
\begin{equation*}
\|x(t)\|_{E_1^0}
\le \|x^0\|_{E_1^0}+  M \eta\| y^0\|_{E_\alpha^s} \int_0^\infty \big(1+\sigma^{-\zeta}\big)e^{-(\omega+\ve/2) \sigma} \,\rd \sigma
\end{equation*}
and thus
\begin{equation}\label{cv}
\|x(t)\|_{E_1^0}\le \|x^0\|_{E_1^0}+\| y^0\|_{E_\alpha^s}\le 2\delta_0\le \frac{r(\eta)}{2}, \qquad t\in [0,T_0).
\end{equation}
Consequently, we infer from~\eqref{yyy}, \eqref{y5}, \eqref{y6}, and~\eqref{cv} that $T_0=t^+$. In particular, all the previous estimates are valid on the interval $(0,t^+)$. Since 
$$
v(t)=u(t)-u_*=x(t)+\phi(x(t))+y(t),\qquad t\in [0,t^+),
$$
they imply  (see also \eqref{n9}) in fact that 
\begin{align}\label{stable4}
u\big([0,t^+)\big)\subset \bar{\mathbb{B}}_{E_\alpha}(u_*, 2 c_{1,\alpha}r(\eta))\subset \bar{\mathbb{B}}_{E_\alpha}( u_*,R)\subset O_\alpha
\end{align}
and
\begin{align*}
u\in {\rm BUC}^\rho([0,t^+),E_{\bar\alpha}).
\end{align*}
Thus, the orbit $u\big([0,t^+))\big)$ is relatively compact in $O_{\bar\alpha}$.
 Invoking  Theorem~\ref{T1}, we conclude that $T_0=t^+=\infty$ 
(as the solutions in $E_{\alpha}$ and $E_{\bar\alpha}$ coincide). 

We note that the argument used when deriving~\eqref{cv} ensures that
\begin{equation}\label{xstar}
\hat x_*:=\lim_{t\to\infty} x(t)=x^0+\int_0^\infty S(x(\sigma),y(\sigma))\,\rd\sigma
\end{equation}
exists in $E_1^0$, since the integral is absolutely convergent. In fact, $\hat x_*\in \mathbb{B}_{E_1^0}(0,r_0)$ and
\begin{equation}\label{25b}
\begin{aligned}
\|x(t)-\hat x_*\|_{E_1^0}&\le \int_t^\infty \|S(x(\sigma),y(\sigma))\|_{E_1^0}\,\rd \sigma 
\leq  M\eta \| y^0\|_{E_\alpha^s} \int_ t^\infty \big(1+\sigma^{-\zeta}\big)e^{-(\omega  +\ve/2)\sigma} \,\rd \sigma\\
&\le M\eta \| y^0\|_{E_\alpha^s}e^{-\omega t} \int_ 0^\infty \big( 1+\sigma^{-\zeta}\big)e^{- \ve\sigma/2} \,\rd \sigma\leq C\| y^0\|_{E_\alpha^s} e^{-\omega t} 
\end{aligned}
\end{equation}
 for $t\ge 0$ and some constant $C>0$,   proving exponential convergence of $x$. 

Additionally, since $\hat x_*\in \mathbb{B}_{E_1^0}(0,r_0)$, we deduce from \eqref{phi2} that
\begin{equation}\label{ustar}
\hat u_*:=u_*+\hat x_*+\phi(\hat x_*)\in \mathcal{E},
\end{equation}
and,  combining~\eqref{n9},~\eqref{yy}, and \eqref{25b}, we find a constant~${C>0}$ such that
\begin{align*}
\|u(t)-\hat u_*\|_{E_\alpha}&\le c_{1,\alpha}\|x(t) -\hat x_*\|_{E_1^0} +c_{1,\alpha}\|\phi(x(t))-\phi(\hat x_*)\|_{E_1^s}+\| y(t)\|_{E_\alpha^s}\\
&\le 2c_{1,\alpha}\|x(t) -\hat x_*\|_{E_1^0}+\| y(t)\|_{E_\alpha^s}\le C\| y^0\|_{E_\alpha^s} e^{-\omega t}
\end{align*}
 for $t\geq 0$.
Noticing that $\| y^0\|_{E_\alpha^s} \le c\|u^0-u_*\|_{E_\alpha}$
according to \eqref{x}-\eqref{y}, this proves~\eqref{des:expstab}.\medskip

Finally, regarding the stability assertion for $u_*$, we observe that, given any $r_1>0$, 
we can choose $\delta>0$ and $2 c_{1,\alpha}r(\eta)\in (0,r_1)$ such that the maximal solution $u(\cdot;u^0)$ with initial value  $u^0\in \bar{\mathbb{B}}_{E_\alpha}(u_*,\delta)$ exists globally and stays in $\bar{\mathbb{B}}_{E_\alpha}(u_*, 2 c_{1,\alpha}r(\eta))$ for all time, see~\eqref{stable4}. This proves Theorem~\ref{MT1}.
\end{proof}

For later use, we state the  following:

 \begin{rem}\label{R2} In view of  \eqref{n9}, \eqref{x},  \eqref{y}, and \eqref{xstar}-\eqref{ustar},  there is a constant~$C>0$ such that
 \begin{equation*}
 \|\hat u_*-u_*\|_{E_1}=\|\hat x_*+\phi(\hat x_*)\|_{E_1}\leq 2\|\hat x_*\|_{E_1^0}\leq C(\|x^0\|_{E_1^0}+\|y^0\|_{E_\alpha^s}) \leq C\|u^0-u_*\|_{E_\alpha}
 \end{equation*}
for $u^0\in \bar{\mathbb{B}}_{E_\alpha}(u_*,\delta)$.
In particular, the equilibrium $\hat u_*=\hat u_*(u^0)$ converges to $u_*$ in $E_1$ as~$u^0\to u_*$ in $E_\alpha$.
 \end{rem}

\section{Proof of Theorem~\ref{MT2}}\label{Sec:3}


A natural framework for studying \eqref{CP} under the assumptions~\eqref{Hypo}--\eqref{Y}  
is provided by time-weighted spaces, which effectively capture the (singular) behavior of  
solutions to~\eqref{CP} near $t=0$.
Let  $E$ be a Banach space,  $\mu\in\R$, and~${T>0}$. Then the time-weighted space~${\rm C}_\mu((0,T],E)$ is defined by
\begin{equation}\label{tws}
 {\rm C}_\mu\big((0,T],E\big):=\big\{u\in {\rm C}((0,T], E)\,:\, \text{$t^\mu  \|u(t)\|_E\to 0$ for $t\to0$}\big\},
 \end{equation}
being a Banach space with the norm
\[
\|u\|_{{\rm C}_\mu((0,T],E)}:=\sup_{t\in(0,T]}t^\mu\|u(t)\|_E. 
\]

The proof of Theorem~\ref{MT2} is based on Theorem~\ref{MT1}, whose application is facilitated by  the following preparatory result. It provides, in particular, uniform estimates for solutions that start in a sufficiently small ball centered at $u_*$. 

\begin{prop}\label{FF}
Assume~\eqref{Hypo}-\eqref{Y} and fix~$\nu\in(\xi,1)$ and~${\eta\in(0,\gamma)}$.
Then there exist constants~$r>0$ and $T>0$  such that $\bar{\mathbb{B}}_{E_\alpha}(u_*,r)\subset O_\alpha$ and  
for each $u^0\in\bar{\mathbb{B}}_{E_\alpha}(u_*,r)$  the Cauchy problem~\eqref{CP}  possesses a unique solution
\begin{equation*}
\begin{aligned}
u(\cdot;u^0)&\in  {\rm C}\big((0,T],E_1\big)\cap{\rm C}^1\big((0,T],E_0\big)\cap {\rm C}\big([0,T],O_\alpha\big)
 \cap  {\rm C}^{\alpha-\beta}\big([0,T],E_\beta\big)\\
 &\quad\cap {\rm C}_{\xi-\alpha}\big((0,T],E_\xi\big).
\end{aligned}
\end{equation*}
Moreover, there is a constant $C>0$ such that, for $u^0, u^1\in \bar{\mathbb{B}}_{E_\alpha }(u_*,r)$, it holds that
\begin{equation}\label{F1}
\begin{aligned}
&\|u(\cdot;u^0)-u(\cdot;  u^1)\|_{{\rm C}([0,T],E_\alpha)} + \|u(\cdot;u^0)-u(\cdot;  u^1)\|_{{\rm C}_{\xi-\alpha}((0,T],E_\xi)}\\
&+ \|u(\cdot;u^0)-u(\cdot;  u^1)\|_{{\rm C}_{\nu-\alpha+\eta}((0,T],E_\nu)}
 \leq  C\|u^0- u^1\|_\alpha.
\end{aligned}
\end{equation}
\end{prop}
\begin{proof}
This is established in \cite[Proposition 2.1 and Corollary 2.2]{MSW25}.
\end{proof}

We are now in a position to prove Theorem~\ref{MT2}.

\begin{proof}[Proof of Theorem~\ref{MT2}]
Under the  assumptions of Theorem~\ref{MT2} we fix $ 0<\omega<\bar\omega<\omega_0$ together with~${\nu\in(\xi,1)}$ and~${\eta\in(0,\gamma)}$.
In view of the embedding~$E_\xi\hookrightarrow E_\beta$, assumption~\eqref{assA} (with $\beta$ replaced by~$\xi$ and~$\alpha$ replaced by~$\nu$), \eqref{E},~\eqref{assB}, and~\eqref{assC} 
(with~$\zeta$ chosen in~$(\nu,1)$) are all satisfied. 
Hence, we may apply  Theorem~\ref{MT1} with~$\omega$ replaced  by~$\bar\omega$   to find constants~$\delta_1>0$ and~$K_1\ge 1$ such that, 
 for each $\|\bar u^0\|_{E_\nu }\leq \delta_1$, the maximal solution~${u=u(\cdot; \bar u^0)}$ to \eqref{CP}  is globally defined, and there exists $\hat u_*=\hat u_*(\bar u^0)\in \mathcal{E}$ with 
 \begin{equation}\label{u2+}
 \|\hat u_*-u_*\|_{E_1}\leq K_1\|\bar u^0-u_*\|_{E_\nu}
 \end{equation}
 by Remark~\ref{R2}, such that  
\begin{equation}\label{stablexx}
\|u(t;\bar u^0)-\hat u_*\|_{E_\nu }\le K_1 e^{-\bar\omega t}\|\bar u^0-u_*\|_{E_\nu },\qquad t\ge 0.
\end{equation}
Let $T> 0$ and $r>0$ be the constants from Proposition~\ref{FF}  and fix an arbitrary $t_*\in (0,T)$.
 Then  we infer from~\eqref{F1} that there exists $\delta\in (0,r)$  and a constant $C>0$ such that for each~$u^0\in \bar{\mathbb{B}}_{E_\alpha }(u_*,\delta)$ 
 we have $\bar u^0:=u(t_*;u^0)\in  \bar{\mathbb{B}}_{E_\nu}(u_*, \delta_1)$ and
\begin{equation}\label{stablex}
\|u(t_*;u^0)-u_*\|_{E_\nu }\le C\|u^0-u_*\|_{E_\alpha}.
\end{equation}
Hence, for such an initial value $u^0\in  \bar{\mathbb{B}}_{E_\alpha }(0,\delta)$,  it follows from $u(t;u^0)=u(t-t_*;u(t_*;u^0))$ for $t\in [t_*,t^+(u^0))$ (by uniqueness of solutions) and $t^+(u(t_*;u^0))=\infty$ that $t^+(u^0)=\infty$.
Moreover, in virtue of~\eqref{stablexx}-\eqref{stablex}, we obtain for~$t\geq t_*$ that
\begin{equation}\label{df}
\begin{aligned}
\|u(t;u^0)-\hat u_*\|_{E_\alpha}&\le C\|u(t;u^0)-\hat u_*\|_{E_\xi } \leq  C \|u(t-t_*;u(t_*;u^0))-\hat u_*\|_{E_\nu }\\
&\le CK_1 e^{-\bar\omega (t-t_*)}\|u(t_*;u^0)-u_*\|_{E_\nu }\\
&\le C  e^{-\bar\omega t}\|u^0-u_*\|_{E_\alpha }. 
\end{aligned}
\end{equation}
Moreover, for $t\in(0,t_*]$, we combine \eqref{F1}, \eqref{u2+}, and \eqref{stablex} to arrive at 
\begin{equation}\label{df1}
\begin{aligned}
&\|u(t;u^0)-\hat u_*\|_{E_\alpha }+t^{\xi-\alpha} \|u(t;u^0)-\hat u_*\|_{E_\xi }\\
&\leq \|u(t;u^0)- u_*\|_{E_\alpha }+t^{\xi-\alpha} \|u(t;u^0)-u_*\|_{E_\xi }+C\|\hat u_*-u_*\|_{E_1}\\
&\leq C(\|u^0-u_*\|_{E_\alpha }+\|u(t_*;u^0)-u_*\|_{E_\nu })\\
&\leq C\|u^0-u_*\|_{E_\alpha }.
\end{aligned}
\end{equation}  
Since $\omega<\bar\omega$, we conclude from~\eqref{df} and \eqref{df1} that there exists  a constant $K\geq 1$  such that, for each~$u^0\in \bar{\mathbb{B}}_{E_\alpha }(u_*,\delta)$, 
the maximal solution~$u(\cdot;u^0)$ to~\eqref{CP} is globally defined and there exists $\hat u_*\in \mathcal{E}$  such that 
\begin{equation*}
\|u(t;u^0)- \hat u_* \|_\alpha +t^{\xi-\alpha}\|u(t;u^0)- \hat u_* \|_\xi\le K e^{-\omega t}\|u^0-u_*\|_\alpha,\qquad t> 0.
\end{equation*}

It remains to prove that $u_*$ is stable. 
But this follows directly from the preceding estimate and from the fact that there is a constant $C>0$ such that,  for each~${u^0 \in \overline{\mathbb{B}}_{E_\alpha}(u_*,\delta)}$, the asymptotic limit 
 $\hat u_*=\hat u_*(u^0)$ satisfies
\[
\|\hat u_* - u_*\|_{E_1} \le C \|u^0 - u_*\|_{E_\alpha}
\]
as established in \eqref{df1}.  This completes the proof of Theorem~\ref{MT2}.
\end{proof}

\section{Examples}\label{Sec:4}

In this section, we illustrate how Theorem~\ref{MT1} can be applied to concrete settings, 
namely the Hele-Shaw problem driven by surface tension and the fractional mean curvature flow, see Examples~\ref{Ex1} and Example~\ref{Ex2}. 
Moreover,  in Example~\ref{Ex3}, we apply Theorem~\ref{MT2} (with $\alpha = \alpha_{\rm crit}$) in the context 
of a parabolic problem exhibiting scaling invariance and a nontrivial semilinearity~$f$.

\subsection{Example (Fractional Mean Curvature Flow)}\label{Ex1} 
The evolution of a  family of periodic graphs~${\{\Sigma(u(t)) : t \ge 0\}}$ in $\mathbb{R}^{n+1}$,
driven by the nonlocal fractional mean curvature of order $\sigma \in (0,1)$, is governed by
\begin{equation}\label{MCF}
    V(t) = - H_\sigma(\Sigma(u(t))), \quad t>0,  \qquad \Sigma(0)=\Sigma_0,
\end{equation}
where $V(t)$ denotes the normal velocity of $\Sigma(u(t))$ and
\begin{align*}
    H_\sigma(\Sigma(u))(x,u(x)):= -\frac{2}{\sigma}\int_{\mathbb{R}^n}
        \frac{u(x)-u(x-y)- y \cdot \nabla u(x-y)}{\big[|y|^2 + (u(x)-u(x-y))^2\big]^{\tfrac{n+1+\sigma}{2}}}\, {\rm d}y, \qquad x \in \mathbb{R}^n,
\end{align*}
is the fractional mean curvature of order $\sigma \in (0,1)$ of the graph~${\Sigma(u)=\{(x,u(x)) : x\in\mathbb{R}^n\}}$. This geometric concept has been introduced quite recently in~\cite{CaffarelliSouganidis10}.
In the  periodic graph setting,  problem~\eqref{MCF} has been formulated in~\cite{MW_2025_Annali}  as a quasilinear  evolution  problem
\begin{equation}\label{QMCF}
u'=A(u)u, \quad t>0,\qquad u(0)=u^0,
\end{equation}
where 
\begin{equation*}
A(u)v(x):=\frac{2}{\sigma}(1+|\nabla u|^2)^{1/2}(x)\int_{\R^n} \frac{v(x)-v(x-y)-y\cdot\nabla v(x-y)}{\big[|y|^2+(u(x)-u(x-y))^2\big]^{\tfrac{n+1+\sigma}{2}}}\,{\rm d}y,\qquad x\in\R^n.
\end{equation*}
A suitable analytic framework for~\eqref{QMCF} is that of  little H\"older spaces ${\rm h}^{s}(\mathbb{T}^n)$, $s\geq0$, defined as the closure of  
${\rm C}^{\infty}(\mathbb{T}^n)$  in the classical H\"older spaces ${\rm C}^{s}(\mathbb{T}^n)$. 
Here,~${\rm C}^{s}(\mathbb{T}^n)$ denotes the Banach space of functions on $\mathbb{R}^n$ which are $2\pi$-periodic in each variable and possess~$(s-[s])$-H\"older continuous derivatives of order 
$[s]:=\max\{k\in\mathbb{N}: k\leq s\}$. 
Moreover, ${\rm C}^{\infty}(\mathbb{T}^n)$ is the intersection of all ${\rm C}^{s}(\mathbb{T}^n)$ with $s\geq 0$. 

As shown in~\cite{MW_2025_Annali}  via a scaling argument, the space ${\rm h}^{1}(\mathbb{T}^n)={\rm C}^{1}(\mathbb{T}^n)$ can be identified, as an 
invariant space for~\eqref{QMCF}. Additionally, it is proven in~\cite{MW_2025_Annali} that problem~\eqref{QMCF} is locally well-posed in the subcritical spaces ${\rm h}^{1+\vartheta}(\mathbb{T}^n)$ 
with $\vartheta\in(0,1)$ arbitrarily close to~$0$. Also in a graph setting, but allowing for 
functions with linear growth that possess gradients in~${\rm C}^{\vartheta}(\mathbb{R}^n)$, $\vartheta>\sigma$, the local well-posedness has been recently established in~\cite{AFW24}.

Concerning stability properties of equilibria to~\eqref{QMCF}, we note that the stationary solutions 
are exactly the constant functions \cite{AFW24, MW_2025_Annali}, which form the one-dimensional linear manifold
\begin{equation}\label{eqex1}
\mathcal{E}=\R.
\end{equation}
We point out that it is not at all clear whether there are flow invariants for~\eqref{QMCF}, and, in particular, whether
 the integral mean 
$$
\langle u\rangle :=\frac{1}{(2\pi)^n}\int_{\mathbb{T}^n} u\,{\rm d}x
$$
of solutions is conserved by the flow. This makes the stability analysis more
intricate.
We therefore rely on Theorem~\ref{MT1} for this purpose.

\begin{thm}\label{T:Ex1}
Let $\sigma,\vartheta\in(0,1)$ be arbitrary and choose $\eta\in(0,1)$ such that
\begin{equation}\label{exponents}
\max\{\sigma,\,\vartheta\}<\eta<\sigma+\vartheta.
\end{equation}
Then, given $u^0\in{\rm h}^{1+\vartheta}(\mathbb{T}^n)$, there exists a unique maximal solution 
$u=u(\,\cdot\,;u^0)$ to \eqref{QMCF} such that
\begin{equation}\label{Ex1:WP}
u\in {\rm C}\big([0,t^+),{\rm h}^{1+\vartheta}(\mathbb{T}^n)\big) \cap {\rm C}\big((0,t^+),{\rm h}^{1+\eta}(\mathbb{T}^n)\big)
   \cap {\rm C}^1\big((0,t^+),{\rm h}^{\eta-\sigma}(\mathbb{T}^n)\big),
\end{equation}
with $t^+=t^+(u^0)\in(0,\infty]$ denoting the maximal existence time.

Moreover, each constant is a stable equilibrium in ${\rm h}^{1+\vartheta}(\mathbb{T}^n)$. 
In addition, there exists a constant $\omega_0>0$ and for each $\omega\in(0,\omega_0)$ there are further constants~$\delta>0$ and~$K\geq1$  such  that, for each 
$u^0\in{\rm h}^{1+\vartheta}(\mathbb{T}^n)$ with
\[
\|u^0-\langle u^0\rangle\|_{{\rm h}^{1+\vartheta}}\leq \delta,
\]
the maximal solution $u=u(\cdot;u^0)$ is globally defined and there is 
$\hat u_*\in\R$ such that 
\begin{equation}\label{Ex1:st}
\|u(t)-\hat u_*\|_{{\rm h}^{1+\vartheta}}\leq K e^{-\omega t}\|u^0-\langle u^0\rangle\|_{{\rm h}^{1+\vartheta}},\qquad t\geq0.
\end{equation}
\end{thm}

\begin{proof}
The local well-posedness of \eqref{QMCF} in the class \eqref{Ex1:WP} follows by applying the 
quasilinear theory from \cite{Am93, MW20} in view of the property (see \cite[Equation~(1.6)]{MW_2025_Annali})
\begin{equation}\label{property}
A\in{\rm C}^\infty\big({\rm h}^{1+\vartheta}(\mathbb{T}^n), \mathcal{H}({\rm h}^{1+\eta}(\mathbb{T}^n),   {\rm h}^{\eta-\sigma}(\mathbb{T}^n))\big),\qquad \max\{\sigma,\,\vartheta\}<\eta<\sigma+\vartheta,
\end{equation}
and the fact that the little H\"older spaces are stable with respect to continuous interpolation, that is,
\begin{equation*}\label{interpol}
({\rm h}^{s_1}(\mathbb{T}^n),{\rm h}^{s_2}(\mathbb{T}^n))_{\theta,\infty}^0
   ={\rm h}^{(1-\theta)s_1+\theta s_2}(\mathbb{T}^n)
\end{equation*}
for $0\leq s_1<s_2$ with $(1-\theta)s_1+\theta s_2\not\in\mathbb{N}$,
where $(\cdot,\cdot)_{\theta,\infty}^0$ is the continuous interpolation functor of 
exponent $\theta\in(0,1)$, see \cite{DG79}. We refer to~\cite{MW_2025_Annali} for details.

We next turn to the exponential convergence stated in~\eqref{Ex1:st}. 
To this end, we use a suitable reformulation of \eqref{QMCF}. We first note that  we may restrict to the equilibrium~${u_*=0}$, since  
\begin{equation}\label{invarex1}
u(\cdot;u^0+c)=u(\cdot;u^0)+c,\qquad c\in\R,\ u^0\in {\rm h}^{1+\vartheta}(\mathbb{T}^n),
\end{equation}
which is a consequence of the fact that
\begin{equation*}
A(u+c_1)[v+c_2]=A(u)[v], \qquad c_1, c_2\in\R.
\end{equation*}
Let $\sigma, \vartheta , \eta\in (0,1) $ satisfy~\eqref{exponents} and define $E_0:= {\rm h}^{\eta-\sigma}(\bT^n)$ and $E_1:= {\rm h}^{1+\eta}(\bT^n)$.  
Choosing exponents~$\vartheta'\in(\eta-\sigma,\vartheta)$ and $\eta'\in(\max\{\vartheta,\,\sigma\},\eta)$, we  infer from   \eqref{exponents} that
\begin{equation}\label{exponents*}
\max\{\sigma,\,\vartheta'\}\leq \max\{\sigma,\,\vartheta\}<\eta'<\eta<\sigma+\vartheta'<\sigma+\vartheta.
\end{equation}
Thus, setting
 \[
0<\beta:=1-\frac{\eta-\vartheta'}{1+\sigma}<\alpha:=1-\frac{\eta-\vartheta}{1+\sigma}<\zeta:=1-\frac{\eta-\eta'}{1+\sigma}<1,
 \]
it follows from \eqref{interpol}, with $E_\theta:=(E_0,E_1)_{\theta,\infty}^0$ for $\theta\in[0,1]$, that
 \begin{equation*} 
E_\zeta={\rm h}^{1+\eta'}(\bT^n)\hookrightarrow E_\alpha={\rm h}^{1+\vartheta}(\bT^n)\hookrightarrow E_\beta={\rm h}^{1+\vartheta'}(\bT^n).
 \end{equation*}
This and \eqref{property}  ensures that assumption \eqref{assA}  holds.
 Moreover, recalling \eqref{eqex1}, also assumption~\eqref{E} is satisfied.
 
 We  then consider the linearized operator $A_*(0)=A(0)\in\kL(E_1,E_0)$, which  is the Fourier multiplier with symbol $m(k)=-\omega_0|k|^{1+\sigma}$, $k\in\mathbb{Z}$, where  $\omega_0>0$ is a positive constant, see~\cite[Section~4]{MW_2025_Annali}. In particular, the kernel of $A_*(0)$ are the constant functions, $0$ is a semi-simple eigenvalue of $A_*(0)$ and 
\[
\sigma(A_*(0))\setminus\{0\}=\big\{-\omega_0|k|^{1+\sigma}\,:\, k\in\mathbb{Z}\setminus\{0\}\big\} \subset(-\infty, -\omega_0],
\] 
 hence also \eqref{assB} is satisfied.
 
 Finally, to verify  assumption \eqref{assC}, we note that the projection corresponding to the spectral set $\{0\}$ of $A_*(0)$ is given by
 \[
Pu:=\langle u\rangle =\frac{1}{(2\pi)^n}\int_{\mathbb{T}^n} u\,{\rm d}x,\qquad u\in E_0.
\]
Moreover, in view of \eqref{property} and \eqref{exponents*}, we  further have
\begin{equation*}
A\in{\rm C}^\infty\big({\rm h}^{1+\vartheta'}(\mathbb{T}^n), \mathcal{H}({\rm h}^{1+\eta'}(\mathbb{T}^n),   {\rm h}^{\eta'-\sigma}(\mathbb{T}^n))\big).
\end{equation*}
Hence, setting $F := \mathrm{h}^{\eta'-\sigma}(\mathbb{T}^n)$ and $G = E_\zeta := \mathrm{h}^{1+\eta'}(\mathbb{T}^n)$, and choosing 
$O_\beta$ to be a sufficiently small neighborhood of $0$ in $E_\beta={\rm h}^{1+\vartheta'}(\mathbb{T}^n)$, this property shows that~\eqref{assC}  holds (with $P\in\kL(F)$ as defined above).

 Consequently, we may  apply Theorem~\ref{MT1} to the evolution problem~\eqref{QMCF} and conclude that $0$ (hence also every other constant solution) is a stable equilibrium of \eqref{QMCF}. 
 Moreover,  given any~$\omega\in(0,\omega_0)$, there exist constants~$\delta>0$ and~$K\geq 1$ such that, for each~${u^0\in \mathbb{B}_{E_\alpha}(0,\delta)}$,  the maximal solution~$u(\cdot;u^0)$ to \eqref{QMCF} is globally defined and there exists $\hat u_*\in\R$ such that
\begin{equation*} 
\|u(t;u^0)-\hat u_*\|_{E_\alpha}\leq K e^{-\omega t}\|u^0\|_{E_\alpha},\qquad t\geq 0.
\end{equation*}
Estimate \eqref{Ex1:st} is now a straightforward consequence of this estimate and of property~\eqref{invarex1}.
\end{proof}


The exponential convergence~\eqref{Ex1:st}  has also been shown in \cite[Theorem~1.3]{MW_2025_Annali} by a direct proof based on the principle of linearized stability from  \cite[Theorem~1.3]{MW20}. However, the proof given herein is considerably less technical.
For a related convergence of bounded global solutions in 
${\rm C}^{1+\vartheta}(\mathbb{T}^n)$ (with $\vartheta>\sigma$) to flat graphs via an energy approach,  we refer to~\cite{CN24}.\medskip

As we have seen in the above proof, the verification of~\eqref{assC} was rather simple due to the given freedom for the choice of the parameters in~\eqref{property}. That is, it was a consequence of the fact that the quasilinear part has a generation property in a scale of interpolation spaces. It is worthwhile to point out that this is a special case of the following general observation, where~\eqref{XXXXX} reflects the situation of~\eqref{property}:

\begin{prop}\label{PPP}
\begin{subequations}\label{XXXXX}
Consider a densely and compactly injected Banach couple $\mathsf{F}_1\hookrightarrow \mathsf{F}_0$ and for  arbitrary admissible interpolation functors
 $(\cdot,\cdot)_\theta$ with $\theta\in [0,1]$ set $\mathsf{F}_\theta:=(\mathsf{F}_0,\mathsf{F}_1)_\theta$. Fix
\begin{equation}\label{410}
0\le\theta_0<\theta_1<\theta_2< \theta_3<\theta_4<\theta_5\le 1,
\end{equation}
let $\mathsf{O}_{\mathsf{F}_{\theta_2}}$ be open in $\mathsf{F}_{\theta_2}$, and set $\mathsf{O}_{\mathsf{F}_{\theta_3}}:=\mathsf{O}_{\mathsf{F}_{\theta_2}}\cap \mathsf{F}_{\theta_3}$. 
Assume that
\begin{equation}\label{AAAA}
\mathsf{A}\in {\rm C}^1\big(\mathsf{O}_{\mathsf{F}_{\theta_3}},\mathcal{H}(\mathsf{F}_{\theta_5},\mathsf{F}_{\theta_1})\big)
\cap {\rm C}^1\big(\mathsf{O}_{{\mathsf{F}}_{\theta_2}},\mathcal{H}({\mathsf{F}}_{\theta_4},{\mathsf{F}}_{\theta_0})\big)
\end{equation} 
\end{subequations}
and
\begin{equation}\label{ffff}
\mathsf{f}\in {\rm C}^1\big(\mathsf{O}_{\mathsf{F}_{\theta_3}},\mathsf{F}_{\theta_1}\big).
\end{equation}
Let $u_*\in \mathsf{F}_{\theta_5}\cap \mathsf{O}_{\mathsf{F}_{\theta_2}}$, set
\begin{equation}\label{A3x}
\mathsf{A}_*:=\mathsf{A}(u_*)+(\partial \mathsf{A}(u_*)[\cdot]) u_* +\partial \mathsf{f}(u_*)\in\mathcal{H}(\mathsf{F}_{\theta_5},\mathsf{F}_{\theta_1}) \cap \mathcal{H}(\mathsf{F}_{\theta_4},\mathsf{F}_{\theta_0}),
\end{equation}
and assume that  $0$ is  a semi-simple eigenvalue of $\mathsf{A}_*\in  \mathcal{L}(\mathsf{F}_{\theta_5},\mathsf{F}_{\theta_1})$. Then  $0$ is also a semi-simple eigenvalue of $\mathsf{A}_*\in \mathcal{L}(\mathsf{F}_{\theta_4},\mathsf{F}_{\theta_0})$.

If
$\mathsf{P}\in \mathcal{L}(\mathsf{F}_{\theta_0}) \cap \mathcal{L}(\mathsf{F}_{\theta_1})$ denotes the spectral projection  onto $\mathrm{ker}(A_*)$ corresponding to the spectral set~$\{0\}$, then
 \eqref{assC} is valid for $(E_0,E_1)=(\mathsf{F}_{\theta_1},\mathsf{F}_{\theta_5})$ and $(F,G)=(\mathsf{F}_{\theta_0},\mathsf{F}_{\theta_4})$.
\end{prop}
 
\begin{proof}
Let us first observe that
$$
\mathsf{A}(u_*)\in\mathcal{H}(\mathsf{F}_{\theta_5},\mathsf{F}_{\theta_1}) \cap \mathcal{H}(\mathsf{F}_{\theta_4},\mathsf{F}_{\theta_0}),\qquad (\partial \mathsf{A}(u_*)[\cdot]) u_* +\partial \mathsf{f}(u_*)\in\mathcal{L}(\mathsf{F}_{\theta_3},\mathsf{F}_{\theta_1}),
$$
 so that the almost reiteration property~\cite[I.Remarks~2.11.2~(a)]{LQPP} together with~\eqref{410} and the perturbation result~\cite[I.Theorem~1.3.1~(ii)]{LQPP} imply~\eqref{A3x}. 

 Next, since  $0$ is  a semi-simple eigenvalue of $\mathsf{A}_*\in  \mathcal{L}(\mathsf{F}_{\theta_5},\mathsf{F}_{\theta_1})$, it follows from~\cite[Proposition~A.2.2]{L95} that 0 is a simple pole of $\lambda\mapsto  \big(\lambda-\mathsf{A}_*\big)^{-1}$ (as a mapping into $\mathcal{L}(\mathsf{F}_{\theta_1})$), or equivalently, that 
\begin{equation}\label{z2}
0=\mathsf{A}_*\mathsf{P}\in \mathcal{L}(\mathsf{F}_{\theta_1})
\end{equation} 
according to~\cite[Proposition~A.2.1]{L95}.  Noticing that $\mathsf{A}_*\in\mathcal{L}(\mathsf{F}_{\theta_5},\mathsf{F}_{\theta_1})$ and  $\mathsf{A}_*\in\mathcal{L}(\mathsf{F}_{\theta_4},\mathsf{F}_{\theta_0})$ both have compact resolvents, their resolvent sets coincide as they consist only of eigenvalues. Hence 0 is also an isolated eigenvalue of $\mathsf{A}_*\in\mathcal{L}(\mathsf{F}_{\theta_4},\mathsf{F}_{\theta_0})$ and it follows that
$$
\mathsf{P}=\frac{1}{2\pi i}\int_{\Gamma} \big(\lambda-\mathsf{A}_*\big)^{-1}\,\rd\lambda \in \mathcal{L}\big(\mathsf{F}_{\theta_0}\big)\cap \mathcal{L}\big(\mathsf{F}_{\theta_1}\big),
$$
where $\Gamma$ is a curve in the resolvent set around 0, see~\cite[Appendix~A]{L95}. 
As $\mathsf{F}_{\theta_1}$ is dense in~$\mathsf{F}_{\theta_0}$, we infer from~\eqref{z2} that $0=\mathsf{A}_*\mathsf{P}\in \mathcal{L}(\mathsf{F}_{\theta_0})$ and therefore 0 is a semi-simple eigenvalue of $\mathsf{A}_*\in \mathcal{L}(\mathsf{F}_{\theta_4},\mathsf{F}_{\theta_0})$ due to~\cite[Proposition~A.2.2]{L95}.

Also note from~\cite[Proposition~A.2.2]{L95} that 
$$
\mathsf{P}\big(\mathsf{F}_{\theta_0}\big)=\mathrm{ker}\Big(\mathsf{A}_*\big\vert_{\mathcal{L}(\mathsf{F}_{\theta_4},\mathsf{F}_{\theta_0})}\Big), \qquad
\mathsf{P}\big(\mathsf{F}_{\theta_1}\big)=\mathrm{ker}\Big(\mathsf{A}_*\big\vert_{\mathcal{L}(\mathsf{F}_{\theta_5},\mathsf{F}_{\theta_1})}\Big).
$$
To identify these kernels, let $x\in \mathsf{F}_{\theta_4}$ satisfy $\mathsf{A}_*x=0$. 
Then  $(\omega-\mathsf{A}_*)x=\omega x\in \mathsf{F}_{\theta_1}$ for any~$\omega\in \Gamma$,  which implies $x\in \mathsf{F}_{\theta_5}$.  
Therefore, the kernels coincide, hence $\mathsf{P}\big(\mathsf{F}_{\theta_0}\big)=\mathsf{P}\big(\mathsf{F}_{\theta_1}\big)$. 
 This proves~\eqref{assC1}, while~\eqref{110b} is a straightforward consequence of~\eqref{AAAA}. 
 The proof of Proposition~\ref{PPP} is thus complete.
\end{proof}

 In Proposition~\ref{PPP} one may replace the requirement of a \textit{compact} embedding $\mathsf{F}_1\hookrightarrow \mathsf{F}_0$ by the assumption that 0 is also an isolated eigenvalue of $\mathsf{A}_*\in  \mathcal{L}(\mathsf{F}_{\theta_4},\mathsf{F}_{\theta_1})$.

We also note that the explicit computation of the projection~$P$ is not required in concrete examples when applying Proposition~\ref{PPP}, see Example~\ref{Ex3}.

\subsection{Example (Hele-Shaw Problem Driven by Capillarity)}\label{Ex2} 
The capillarity-driven Hele-Shaw problem~\cite{HS98} models the two-dimensional motion of an incompressible fluid blob confined 
to a narrow channel between two parallel plates. 
Let $\Omega(t) \subset \mathbb{R}^2$ be the bounded domain occupied by the fluid at time  $t \ge 0$ and let $\Gamma(t) := \partial \Omega(t)$ be its boundary. 
The evolution is then governed by the system
\begin{subequations}\label{HSP}
\begin{equation}\label{HSP1}
\left.
\arraycolsep=1.4pt
\begin{array}{rcllll}
\Delta u(t)&=&0&\quad\text{in $\Omega(t)$,}\\
u(t)&=& \kappa_{\Gamma(t)}&\quad\text{on $\Gamma(t)$,}\\
V(t)&=&-\partial_{{\rm n}_{\Gamma(t)}} u(t)&\quad\text{on $\Gamma(t)$}
\end{array}
\right\}
\end{equation}
for $t>0$.
Here, $u(t)$ denotes the fluid pressure. 
The velocity field  is assumed to obey Darcy’s law~\cite{Da56} and has been eliminated from the system using the incompressibility property. 
Furthermore, $V(t)$ denotes the normal velocity of $\Gamma(t)$, ${\rm n}_{\Gamma(t)}$ its outer unit normal, and $\kappa_{\Gamma(t)}$ its curvature. 
In~\eqref{HSP1}, all material parameters are normalized to~$1$.
Additionally, we impose the initial condition
\begin{equation}\label{HSP2}
\arraycolsep=1.4pt
\begin{array}{rcl}
\Omega(0)&=&\Omega_0.
\end{array}
\end{equation}
\end{subequations}

An important feature of the  flow \eqref{HSP} is the fact that it preserves both the area and the center of mass of the fluid domain, that is,
\begin{align}
\frac{{\rm d}}{{\rm d} t}|\Omega(t)| &= -\int_{\Gamma(t)} \partial_{{\rm n}_{\G(t)}} u(t)\,  {\rm d}\Gamma = 0, \label{lol1}\\
-\frac{{\rm d}}{{\rm d} t}\int_{\Omega(t)} z\,{\rm d}z &=\int_{\Gamma(t)} z\partial_{{\rm n}_{\G(t)}} u(t)\, {\rm d}\Gamma= \int_{\Gamma(t)} u(t){\rm n}_{\G(t)}\, {\rm d}\Gamma 
=\int_{\Gamma(t)} \kappa_{\Gamma(t)} {\rm n}_{\G(t)}\, {\rm d}\Gamma = 0,\label{lol2}
\end{align}
where $z=(x,y)$ is the variable in $\R^2$; see e.g. \cite{NT98, MW2025x}.
By a maximum–principle argument, one may identify the stationary solutions to \eqref{HSP} 
as a three–dimensional smooth manifold~$\mathcal{E}$ consisting exclusively of circles (with varying radii and centers).\medskip

A broad range of analytical techniques has been developed to study the stability of circular equilibria for \eqref{HSP} as well as in related geometric free–boundary problems. 
For the two-phase Hele–Shaw (Muskat) setting, \cite{CP93} established convergence to a circle for small analytic perturbations by employing complex–analytic methods together with energy estimates.
 Based on a center–manifold approach, stability results were obtained in Hölder-type settings: 
\cite{ES98}  treated perturbations of spheres in the little Hölder spaces ${\rm h}^{2+\alpha}(\Sigma)$ in the context of the Mullins–Sekerka problem (where $\Sigma$ is a sphere in $\mathbb{R}^n$, $n\geq 2$), 
while \cite{EMM10} used this approach in a non-Newtonian Hele–Shaw flow in the regularity class ${\rm h}^{4+\alpha}(\mathbb{T})$.
In the framework of the Mullins–Sekerka problem, \cite{PSZ9a, PS16} used  a generalized principle of linearized stability based on $L_p$–maximal regularity to prove
exponential stability for perturbations in $W^{4-4/p}_p(\Sigma)\hookrightarrow {\rm C}^2(\Sigma)$, $p>n+2$.
The stability of Muskat bubbles in (critical) Wiener spaces on the torus~$\mathbb{T}$ was addressed in \cite{GGPS25,GGPS23} using potential theory and energy methods.

We now show that the stability of equilibria for~\eqref{HSP} can be deduced from Theorem~\ref{MT1} in the scale of Bessel potential spaces $H^{r}(\mathbb{\bT})$ over the torus~$\bT$ 
in a framework of (almost) optimal regularity. 
Specifically, by exploiting scaling invariance, one identifies~$r = 3/2$ as the critical exponent for \eqref{HSP} within the scale of Bessel  potential spaces~\cite{MW2025x}.\medskip

To set the stage, we introduce
\[
 \mathcal{V}_r:=\{\rho\in H^{ r}(\mathbb{T}) : \rho>0\}, \qquad r\in(3/2,2],
\]
and define, for a given function $\rho\in\mathcal{V}_r$, the periodic curve 
\[
\Gamma_\rho := \{ \rho(\tau)\,e^{i\tau} : \tau \in \mathbb{R} \}
\]
enclosing a bounded star–shaped domain $\Omega_\rho\subset\mathbb{R}^2$.
Assuming that $\Omega(t)=\Omega_{\rho(t)}$ for~${t>0}$ and using  potential theory and tools from harmonic analysis, the evolution problem~\eqref{HSP} is reformulated in~\cite{MW2025x} as a quasilinear problem for $\rho$ of the form
\begin{equation}\label{intrQPP}
\rho' = A(\rho)\rho, \quad t>0, \qquad \rho(0)=\rho^0,
\end{equation}
where 
\begin{equation}\label{regP:Ex2}
A \in {\rm C}^\infty\bigl(\mathcal{V}_r,\, \mathcal{L}(H^{r+1}(\mathbb{T}),\, H^{r-2}(\mathbb{T}))\bigr),\qquad r\in(3/2,2].
\end{equation}

Focusing on the stability of the circle centered at the origin  of radius~$1$, that is, the stationary solution $\rho_*=1$ (a similar argument yields exponential stability for any circle of radius $r>0$ centered at $z\in\R^2$), it is natural --~in view of the invariants~\mbox{\eqref{lol1}--\eqref{lol2}}~--~to consider the asymptotic behavior of solutions satisfying the identities
\begin{equation}\label{eqcons}
\int_{-\pi}^{\pi} (\rho^2 - 1)\, {\rm d}s = 0, \qquad
\int_{-\pi}^{\pi} \rho^3 \cos s\, {\rm d}s = \int_{-\pi}^{\pi} \rho^3 \sin s\, {\rm d}s = 0,
\end{equation}
which simply express that the fluid blob has area $\pi$ (the first identity) and that its center of mass is at $(0,0)$ (the second and third identities).
Applying Theorem~\ref{MT1} we establish  the following result:

\begin{thm}\label{T:Ex2}
Let $3/2<\bar r<r<2$. Then problem~\eqref{intrQPP} has for each  $\rho^0\in \mathcal{V}_{r}$
 a unique maximal solution~$\rho=\rho(\cdot;\rho^0) $ such that 
\begin{equation}\label{prop1}
\rho\in {\rm C}\big([0,t^+),\mathcal{V}_{r}\big)\cap {\rm C}\big((0,t^+), H^{\bar r+1}(\bT)\big) \cap {\rm C}^1\big((0,t^+), H^{\bar r-2}(\bT)\big)
\end{equation}
and, for some $\eta\in(0,(r-\bar r)/3]$,
\begin{equation}\label{prop1'}
\rho\in{\rm C}^{\eta}\big([0,t^+), H^{\bar r}(\bT)\big),
\end{equation}
where $t^+=t^+(\rho^0)>0$ is the maximal existence time of the solution.

Moreover, $\rho_*=1$ is a stable equilibrium in $ H^{r}(\bT)$, and, for any  given $\omega\in (0,6)$, there are constants~$\delta>0$  and~${K\geq 1}$ such that for each $\rho^0\in\mathbb{B}_{H^r}(1,\delta)$ 
 satisfying  \eqref{eqcons},  the maximal solution~$\rho(\cdot;\rho^0)$  to \eqref{intrQPP} is globally  defined and
\begin{equation}\label{convEx2}
\|\rho(t;\rho^0)-1\|_{H^{{ r}}}\leq Ke^{-\omega t}\|\rho^0-1\|_{H^{{r}}},\qquad t\geq0.
\end{equation}
\end{thm}


\begin{proof}
The well-posedness of  \eqref{intrQPP} is established in
 \cite[Theorem 1.1]{MW2025x} relying on~\eqref{regP:Ex2} and abstract quasilinear parabolic theory (see Theorem~\ref{T1}). We omit details here.

As for the stability statement, we set $ E_0:=H^{\bar r-2}(\bT)$ and $E_1:= H^{\bar r+1}(\bT)$ and choose 
\[
0<\beta:=\frac{2}{3}<\alpha:=\frac{2+r-\bar r}{3}<\zeta:=\frac{3+r'-\bar r}{3} <1\qquad\text{with}\qquad 3/2<r'<\bar r<r<2.
\]
We recall  the interpolation property 
  \begin{align}\label{IP}
[H^{r_0}(\mathbb{\bT}),H^{r_1}(\mathbb{\bT})]_\theta=H^{(1-\theta)r_0+\theta r_1}(\mathbb{\bT}),\qquad\theta\in(0,1),\ \,  -\infty\leq r_0\leq r_1<\infty,
\end{align}
where $[\cdot,\cdot]_\theta$ refers to the complex interpolation functor of exponent $\theta\in(0,1)$.
Setting $E_\theta:=[E_0,E_1]_{\theta}$ for $\theta\in[0,1]$, we infer from~\eqref{IP} that
 \begin{equation*}
E_\zeta=H^{1+r'}(\bT)\hookrightarrow E_\alpha=H^{r}(\bT)\hookrightarrow E_\beta=H^{\bar r}(\bT),
 \end{equation*}
and \eqref{regP:Ex2} ensures that $\Phi \in {\rm C}^\infty\big(O_\beta, \mathcal{L}(E_1,\, E_0)\big)$ with $O_\beta:=\mathcal{V}_{\bar r}$.
Therefore, we have verified assumption~\eqref{assA}.

Concerning the set of equilibria close to the unit circle, we follow~\cite{ES98} and solve the equation
\[
|\rho e^{i\tau} - (a,b)|^2 = r^2
\]
for $\rho\in (0,\infty)$. This yields that, for $(a,b,r)$ in a sufficiently small neighborhood 
$U$ of~$0$ in~$\mathbb{R}^3$, the circle with radius $1+r$ and center $(a,b)$  coincides with the curve $\Gamma_\rho$  for $\rho$ given  by 
\[
\rho
= a\cos  + b\sin  
+ \sqrt{(a\cos  + b\sin )^2 + (1+r)^2 - a^2 - b^2}=:\Psi(a,b,r)\in E_1.
\]
Then $\Psi\in {\rm C}^\infty(U,E_1)$ satisfies $\Psi(0)=1$ and  
\begin{equation}\label{ggv}
\mathrm{rg}(\partial \Psi(0))
= \mathrm{span}\{1,\, \sin,\, \cos\},
\end{equation}
which shows that~\eqref{E}  holds.

To verify~\eqref{assB}, we  note from \cite[Section~5]{MW2025x}  that
\[
A_*(0)=A(1)=H\circ\Big(\frac{{\rm d}^3}{{\rm d}\tau^3}+\frac{{\rm d}}{{\rm d}\tau}\Big),
\]
where $H$ denotes the periodic Hilbert transform. 
In particular, $A_*(0)\in\kL(E_1,E_0)$ is the Fourier multiplier with symbol~$m(k):=|k|(1-|k|^2)$, $k\in\mathbb{Z}$.  That is,
\[
A_*(0)\rho=\sum_{k\in\mathbb{Z}} m(k)a_k e^{ik\tau},\qquad \rho=\sum_{k\in\mathbb{Z}} a_k e^{ik\tau}\in E_1.
\]
Hence,
\[
\sigma(A_*(0))=\{|k|(1-|k|^2)\,:\, k\in\mathbb{Z}\}
\]
 consists solely of isolated eigenvalues. 
Together with~\eqref{ggv}, this verifies~\eqref{assB}.  We also note  that  $-\omega_0=\sup\{|k|(1-|k|^2)\,:\, |k|\geq 2\}=-6$.
 
Finally,  regarding the  last assumption \eqref{assC} of Theorem~\ref{MT1}, we note that the projection~$P$ corresponding to the spectral set $\{0\}$ of $A_*(0)$ is given by
 \[
P \rho =\sum_{|k|\leq 1}  a_k e^{ik\tau},\qquad \rho=\sum_{k\in\mathbb{Z}} a_k e^{ik\tau}\in E_1.
\]
Therefore, setting $G:= E_\zeta=H^{1+r'}(\mathbb{T})$ and $F=H^{r'-2}(\mathbb{T})$, we may infer from 
\eqref{regP:Ex2} (with~$r$ replaced by~$r'$), after possibly making $O_\beta$ smaller,
 that~\eqref{assC} is fulfilled. 
 
We are thus in a position to apply Theorem~\ref{MT1} to \eqref{intrQPP} and establish the stability of the equilibrium $\rho_*=1$ and
the desired estimate~\eqref{convEx2} by using~\eqref{eqcons}.
\end{proof}


The  convergence~\eqref{convEx2}  has also been shown in \cite[Theorem~1.4]{MW2025x} by applying an abstract stability result \cite[Theorem~A.1]{MW2025x} and exploiting the flow invariants~\mbox{\eqref{lol1}–\eqref{lol2}}. However, the proof given herein is considerably shorter and less technical.
We point out that, in this example as well, assumption~\eqref{assC} alternatively follows from Proposition~\ref{PPP} in view of~\eqref{regP:Ex2}.

\subsection{Example (A Quasilinear Problem with Scaling Invariance)}\label{Ex3} 
In this example we shall apply Theorem~\ref{MT2} to the quasilinear evolution equation \cite{PSW18, QS19}
\begin{subequations}\label{E3}
\begin{equation}\label{E3a}
 \partial_tu={\rm div}(a(u)\nabla u) +|\nabla u|^{\kappa}\qquad \text{in } \Omega,\ \, t>0, 
\end{equation}
subject to homogeneous Neumann boundary conditions
\begin{equation}\label{E3b}
  \partial_\nu u=0\qquad \text{on } \partial\Omega,\ \, t>0, 
\end{equation}
and initial condition
\begin{equation}\label{E3c}
u(0)=u^0 \qquad \text{in $\Omega  $,} 
\end{equation}
\end{subequations}
where ${\kappa>3}$, $a\in {\rm C}^{3-}(\R)$ is a  strictly positive   function, and $u^0:\Omega\to\R$ is a given function on a smooth and bounded domain $\Omega\subset\R^n$  with $n\geq 1$.
 For $p\in(1,\infty)$ we set
\begin{equation*}
 H_{p,N}^{2\theta}(\Omega)
 :=\left\{\begin{array}{ll} \{u\in H_{p}^{2\theta}(\Omega) \,:\,  \partial_\nu u=0 \text{ on } 
 \partial\Omega\}, &1+\frac{1}{p}<2\theta\leq 2 ,\\[3pt]
	 H_{p}^{2\theta}(\Omega), & -1+\frac{1}{p}< 2\theta< 1+\frac{1}{p},\end{array} \right.
\end{equation*}
and note, taking $a$ to be constant and using scaling invariance arguments, that the space~$H^{s_c}_{p,N}(\Omega)$ with
\begin{equation}\label{crsp}
s_c:=\frac{n}{p}+\frac{\kappa-2}{\kappa-1}
\end{equation}
is a critical space for  \eqref{E3}, see~\cite{PSW18, MRW25}. Local well-posedness of \eqref{E3} in critical scaling invariant spaces has been addressed  in \cite[Example 2]{PSW18} and subsequently in~\cite[Example~4]{MRW25}.
Moreover, the  exponential stability  of the zero solution to \eqref{E3} with~\eqref{E3b} replaced by a homogeneous Dirichlet boundary condition (in this case is $0$ the only constant solution to \eqref{E3a})
has been established in the corresponding critical Bessel potential space of exponent~\eqref{crsp}   in~\cite[Example 3]{MSW25}.
For homogeneous Neumann boundary conditions, all constants are equilibrium solutions to~\eqref{E3}, which makes the stability analysis more intricate. 
The following theorem, which is new to the best of our knowledge, establishes the   stability of the constant solutions to \eqref{E3} in the critical space identified by \eqref{crsp}.
For simplicity, in Theorem~\ref{T:Ex3} we consider the asymptotic behavior of solutions that are initially close to $u_*=0$ in this phase space (though $0$ may be replaced by any other constant).

\begin{thm}\label{T:Ex3}
Let $\kappa>3$,  $p\in (2n,(\kappa-1)n)$  with $p \neq (n-1)(\kappa-1)$, and~$\tau\in(0,1)$ such that
\begin{equation*}
\frac{1}{2}<2\tau<1-\frac{n}{p}. 
\end{equation*}
Further set
\begin{equation*}
 0<\bar s:=2\tau+ \frac{n}{p}<s_c<s:=1+\frac{n(\kappa-1)}{p\kappa}<  2-2\tau
\end{equation*}
and
\begin{equation*}
 \mu:=\frac{1}{2(\kappa-1)}-\frac{n}{2p\kappa}\in(0,1),\qquad  \vartheta:=\frac{\kappa-2}{2(\kappa-1)}-\tau\in(0,1).
\end{equation*}
Then, given any $u^0\in H^{s_c}_{p,N}(\Omega)$, there exists a unique maximal solution $u=u(\cdot; u^0)$ to~\eqref{E3}  such that
\begin{subequations}\label{regv3}
\begin{equation}
\begin{aligned}
u&\in {\rm C}\big((0,t^+), H^{2-2\tau}_{p,N}(\Omega)\big)\cap {\rm C}^1\big((0,t^+), H^{-2\tau}_{p}(\Omega)\big)\cap {\rm C}\big([0,t^+),H^{s_c}_{p,N}(\Omega)\big)\\
&\quad\cap  {\rm C}^{\vartheta}\big([0,t^+), H^{\bar s}_{p}(\Omega)\big) 
\end{aligned}
\end{equation}
and
\begin{equation}
\lim_{t\to 0} t^{\mu}\|u(t)\|_{H^s_p} = 0,
\end{equation}
\end{subequations}
where $t^+:=t^+(u^0)\in(0,\infty] $ is the maximal existence time of the solution.

  Moreover, $u_*=0$ is a stable equilibrium in $H^{s_c}_{p,N}(\Omega)$, and  there exists $\omega_0>0$ and for each $\omega\in (0,\omega_0)$ there exist constants $\delta>0$ and  $K\geq 1$ such that for all $\|u^0\|_{H^{s_c}_{p}}\leq \delta$ the solution~${u=u(\cdot; u^0)}$ to~\eqref{E3}
    is globally defined and there is $\hat u_*\in\R$ such that 
 \begin{equation}\label{estE3}
  \|u(t;u^0)-\hat u_*\|_{H^{s_c}_{p}}+t^{\mu}\|u(t;u^0)-\hat u_*\|_{H^{s}_p}\le K\, e^{-\omega t}\,\|u^0\|_{H^{s_c}_p},\qquad t>0.
 \end{equation}
\end{thm}

Before providing the proof of Theorem~\ref{T:Ex3} we remark the following:

\begin{rem}\label{R:Ex3} 
The integral mean of any solution to~\eqref{E3} is a strictly increasing function of time and its limit is the asymptotic limit $\hat u_*$ from~\eqref{estE3}.
\end{rem}

\begin{proof}[Proof of Theorem~\ref{T:Ex3}] We proceed similarly as in \cite{MRW25} and set
\[
 F_0:=L_p(\Omega),  \qquad  F_1:= W_{p,N}^{2}(\Omega)=H_{p,N}^{2}(\Omega),
 \] 
 recalling from  \cite[\S 4]{Am93} that 
\[
B_0:=\Delta_N:=\Delta|_{W_{p,N}^{2}(\Omega)}\in \mathcal{H}\big(W_{p,N}^{2}(\Omega),L_p(\Omega)\big).
\]
Let
$$
\big\{(F_\theta,B_\theta)\,:\, -1\le \theta<\infty\big\}
$$
be the interpolation-extrapolation scale generated by $(F_0,B_0)$ and the 
complex interpolation functor~$[\cdot,\cdot]_\theta$ (see \cite[\S 6]{Am93} and \cite[\S V.1]{LQPP}), that is,
\begin{equation*}
B_\theta\in \mathcal{H}(F_{1+\theta},F_\theta),\qquad -1\le \theta<\infty,
\end{equation*}
with (see \cite[Theorem~7.1; Equation (7.5)]{Am93})
\begin{equation}\label{f2N}
 F_\theta\doteq H_{p,N}^{2\theta}(\Omega),\qquad   2\theta\in\Big(-1+\frac{1}{p},2\Big]\setminus\Big\{1+\frac{1}{p}\Big\}.
\end{equation}
Moreover, since
$\Delta_N-1$ has bounded imaginary powers (see~\cite[III.~Examples 4.7.3~(d)]{LQPP}), we infer from \cite[Remarks~6.1~(d)]{Am93} that
\begin{equation}\label{f3N}
 [F_\beta,F_\alpha]_\theta\doteq F_{(1-\theta)\beta+\theta\alpha},\qquad -1\leq \beta<\alpha,\ \theta\in(0,1).
\end{equation}
Define now
\[
E_\theta:=   H^{2\theta-2\tau}_{p,N}(\Omega),\qquad 2\tau+1+\frac{1}{p}\neq 2\theta\in[0,2].
\]
We further set $q := \kappa > 3$ and 
\begin{equation}\label{bnnn1}
0 < \gamma := \tau < \beta := \tau + \frac{\bar s}{2}< \alpha := \tau + \frac{s_c}{2}< \xi := \tau + \frac{s}{2} < 1.
\end{equation}
Then
\[
\vartheta = \alpha - \beta, \qquad \mu = \xi - \alpha,
 \qquad\alpha_{\rm crit}= \frac{\kappa \xi - 1 - \gamma}{\kappa - 1}= \alpha\in (\beta, \xi),
\]
and, noticing that none of the constants $\bar s$, $s$, or $s_c$ fixed in Theorem~\ref{T:Ex3}  is equal to~$1+1/p$, we have
\[
E_\xi = H^{s}_{p,N}(\Omega)\hookrightarrow E_\alpha = H^{s_c}_{p,N}(\Omega)
\hookrightarrow E_\beta = H^{\bar s}_{p,N}(\Omega)\hookrightarrow E_\gamma= L_{p}(\Omega).
\]

The well-posedness of \eqref{E3} has been obtain in \cite[Example 3]{MRW25}  by noticing that  \eqref{E3} can be reformulated as a   quasilinear  evolution problem
\begin{equation}\label{QEE2}
u'=A(u)u+f(u),\quad t>0,\qquad u(0)=u^0,
\end{equation} 
where $A: H^{\bar s}_{p,N}(\Omega)\to\kL(H^{2-2\tau}_{p,N}(\Omega),H^{-2\tau}_{p,N}(\Omega))$ is defined by 
\[
A(u)v:={\rm div} (a(u)\nabla v),\qquad  u\in H^{\bar s}_{p,N}(\Omega),\ v\in H^{2-2\tau}_{p,N}(\Omega),
\]
and $f: H^{ s}_{p,N}(\Omega)\to L_p(\Omega)$  is defined by
\[
f(u):=|\nabla u|^{\kappa},\qquad u\in H^{ s}_{p,N}(\Omega).
\] 

With respect to the semilinearity $f$, we observe that $H^{s-1}_{p,N}(\Omega)\hookrightarrow L_{p\kappa}(\Omega)$
and that the Nemytskii operator $w \mapsto |w|^\kappa$
belongs to~${\rm C}^1\big(L_{p\kappa}(\Omega),L_{p}(\Omega)\big)$,
which follows from standard arguments based on the estimate
\[
\big|\,x|x|^\alpha - y|y|^\alpha\big|\le (1+\alpha)\big(|x|^\alpha + |y|^\alpha\big)|x-y|,\qquad x,y\in\mathbb{R},\ \alpha>-1.
\]
Combining these properties, we readily conclude that
\begin{equation*}\label{3regf}
f \in {\rm C}^1\big(H^{s}_{p,N}(\Omega),L_p(\Omega)\big).
\end{equation*}
Moreover, there exists a constant $C>0$ such that
\begin{align*}
\|f(u)-f(v)\|_{L_p}\leq C (\|u\|_{H^{s}_{p}}^{\kappa-1}+\|v\|_{H^{s}_{p}})^{\kappa-1}\|u-v\|_{H^{s}_{p}},\qquad v, w\in H^{s}_{p,N}(\Omega).
\end{align*}
Concerning the quasilinear term $u \mapsto A(u)$, we observe that since 
$a \in {\rm C}^{3-}(\mathbb{R})$, standard arguments similar to those in 
\cite[Lemma 4.1]{MW_PRSE} ensure that the Nemytskii operator 
$u \mapsto a(u)$ belongs to ${\rm C}^1\big(L_\infty(\Omega)\cap W^r_p(\Omega),\, L_\infty(\Omega)\cap W^r_p(\Omega)\big)$ for each $r\in(0,1)$.
Let $2\varepsilon \in (0,4\tau-1)$ be fixed and recall from 
\cite[Equations (5.2)–(5.6)]{Am93} the embeddings
\[
H^{\bar s}_{p}(\Omega)\hookrightarrow  W^{\bar s-\varepsilon}_{p}(\Omega)\hookrightarrow  H^{\bar s-2\varepsilon}_{p}(\Omega)\hookrightarrow  H^{1-2\tau}_{p}(\Omega),
\]
all these spaces being algebras with respect to pointwise multiplication since $1-2\tau>n/p$.
It now follows that the Nemytskii operator $u \mapsto a(u)$ also belongs to~${\rm C}^1\big(H^{\bar s}_p(\Omega),\, H^{\bar s-2\varepsilon}_p(\Omega)\big)$ and
\begin{equation*}
A \in {\rm C}^1\big(H^{\bar s}_{p,N}(\Omega),
\mathcal{L}\big(H^{2-2\tau}_{p,N}(\Omega),H^{-2\tau}_{p,N}(\Omega)\big)\big).
\end{equation*}
Moreover, noticing for $u\in H^{\bar s}_{p,N}(\Omega)$ that the function $a(u)$  lies in~$H^{\bar s-2\ve}_{p,N}(\Omega)\hookrightarrow {\rm C}^\rho(\overline{\Omega})$ with~$\rho=2(\tau-\ve)>1-2\tau$, and is also strictly positive, we infer from \cite[Theorem 8.5]{Am93} that $A(u)\in \mathcal{H}\big(H^{2-2\tau}_{p,N}(\Omega),H^{-2\tau}_{p,N}(\Omega)\big)$. 
Consequently,
\begin{equation}\label{3rega}
A\in{\rm C^{1}}\big(H^{\bar s}_{p,N}(\Omega),\mathcal{H}(H^{2-2\tau}_{p,N}(\Omega),H^{-2\tau}_{p,N}(\Omega))\big).
\end{equation}
Therefore, we checked~\eqref{Hypo} with $q=\kappa$ and~\eqref{Y} as well. 

Next, it readily follows from~\eqref{E3a}
 and~\eqref{E3b} that the equilibria are the constants, i.e. $\mathcal{E}=\R$, hence~\eqref{E} holds.
In order to verify~\eqref{assB} we observe that 0 is a simple eigenvalue of $$A_*(0)=A(0)=a(0)\Delta_N\in \mathcal{L}\big(W_{p,N}^2(\Omega),L_p(\Omega))\big)$$ and a similar argument as in the proof of Proposition~\ref{PPP} implies that 0 is a simple eigenvalue of $A_*(0)\in \mathcal{L}\big(H_{p,N}^{2-2\tau}(\Omega),H_{p,N}^{-2\tau}(\Omega)\big)$. Condition~\eqref{assC} then follows from Proposition~\ref{PPP} and~\eqref{3rega} with $\tau$ replaced by some $\tau'\in (1/4,\tau)$.

Consequently, we may invoke Theorem~\ref{MT2} to derive Theorem~\ref{T:Ex3}.
\end{proof}

\bibliographystyle{siam}
\bibliography{Literature}

\end{document}